\documentclass[12pt]{amsart}

\usepackage{amsfonts}
\usepackage{amssymb}
\usepackage{amsxtra}
\usepackage{amsthm}
\usepackage{amsfonts, amscd}
\usepackage{xypic, xy}

\def\ZZ{{\mathbb Z}}

\def\CC{{\mathbb C}}
\def\AA{{\mathbb A}}
\def\RR{{\mathbb R}}

\def\PP{{\mathbb P}}

\def\cA{\mathcal{A}}
\def\cB{\mathcal{B}}
\def\cG{\mathcal{G}}
\def\cH{\mathcal{H}}
\def\cC{\mathcal{C}}
\def\cL{\mathcal{L}}
\def\cO{\mathcal{O}}
\def\cM{\mathcal{M}}

\def\cN{\mathcal{N}}
\def\cU{\mathcal{U}}
\def\cT{\mathcal{T}}
\def\cX{\mathcal{X}}

\DeclareMathOperator{\Coker}{Coker} 
\DeclareMathOperator{\Bl}{Bl}
 
\DeclareMathOperator{\Hom}{Hom}

\DeclareMathOperator{\rank}{rank}

\DeclareMathOperator{\Spec}{Spec}
\DeclareMathOperator{\Stab}{Stab}
\DeclareMathOperator{\Proj}{Proj}

\newtheorem{lemma}{Lemma}[section]
\newtheorem{theorem}[lemma]{Theorem}

\newtheorem{proposition}[lemma]{Proposition}

\theoremstyle{definition}
\newtheorem{definition}[lemma]{Definition}

\newtheorem{remark}[lemma]{Remark}
\newtheorem{notation}{Notation}

\numberwithin{equation}{section}

\newcommand{\bean}{\begin{eqnarray}}
\newcommand{\eean}{\end{eqnarray}}
\newcommand{\be}{\begin{displaymath}}
\newcommand{\ee}{\end{displaymath}}
\newcommand{\bea}{\begin{eqnarray*}}
\newcommand{\eea}{\end{eqnarray*}}

\newcommand{\ol}{\overline}

\begin{document}

\title[A virtual  class of the moduli space of $(\CC^*)^n$-equivariant morphisms]{A virtual fundamental class construction for the moduli space of $(\CC^*)^n$-equivariant morphisms}

\author{Andrei~Musta\c{t}\v{a}}
\address{School of Mathematical Sciences, University College Cork, Ireland}
\email{{\tt andrei.mustata@ucc.ie}}

\date{\today}

\begin{abstract}
Let X be  a smooth projective variety with the action of $(\CC^*)^n$.  The article describes the moduli space of  $(\CC^*)^n$ equivariant morphisms from stable toric varieties into X as  the inverse limit of the GIT quotients of $X$ and their flips when these spaces are enhanced with a naturally associated  Deligne-Mumford stack structure. This description is used for constructing a class in the Chow group of the moduli space of dimension $dim(X)-n$ which is invariant under equivariant deformations of $X$.

\end{abstract}

\maketitle

\bigskip

\section{Introduction}
In \cite{alex} V.Alexeev introduced a moduli space of $(\CC^*)^n$ -- equivariant morphisms from stable toric varieties into the projective space. More generally, one can extend to the case when the target is a smooth projective variety.
    The aim of this article is to construct a virtual fundamental class for the moduli space of equivariant morphisms of stable toric varieties into $X$. However the definition of a virtual class as presented in \cite{behrend} requires the existence of a perfect obstruction theory, condition which is not satisfied by these moduli spaces. This article doesn't intend to give a general framework for dealing with the absence of a perfect obstruction theory. It merely uses ad-hoc methods for constructing a class of the correct dimension that is invariant under equivariant deformations of $X$ and enjoys good functorial properties. We will interpret our definition in terms of the construction in \cite{behrend}. In particular, in the case when $n=1$ this moduli space is a union of components in the fixed locus of a moduli space of stable maps into $X$. In this case we will check that our definition coincides with the definition from \cite{pan} (see Theorem 6.3). In joint work with Anca Mustata, the author intends to use this result for the computation of Gromov-Witten invariants of varieties with $\CC^*$ action. For a higher dimensional torus one might view any of these moduli spaces as a union of components in the fixed locus in certain moduli spaces of stable maps that have yet to be constructed (for details see Section$7$).

In \cite{alex} V.Alexeev showed that the reduced structure of the moduli space of equivariant morphisms from stable toric varieties into a projective space has a morphism into the inverse limit of GIT quotients of the projective space by the torus action. We consider a natural stack structure on the GIT quotients and on flips between them. By abuse of notation, from now on we will call the inverse limit of these stacks the inverse limit of GIT quotients and prove that it is isomorphic with the moduli space constructed in \cite{alex}. Notice that the inverse limit doesn't commute with the operation of passing to the coarse moduli space, but there is a morphism between the coarse moduli space of the inverse limit and the inverse limit of coarse moduli spaces. A simple example for the relation between the two notions is given by the relation between the coarse moduli space of a stack and the coarse moduli space of its inertia stack. The virtual class we construct takes into account only the components of the moduli space which have dimension larger than expected. Comparing our moduli space with an inertia stack, it might be tempting to inquire if there is an analog of the stringy cohomology which also takes into account the components of dimension smaller than expected, but this problem will not be addressed in the present article.

To construct the universal family over the inverse limit of GIT quotients, we will use the moduli space of stable maps from genus $0$ curves into $X$, thus underlying  the connection between the two.

The paper is organized as follows: In Section $2$ we will define the inverse limit of GIT quotients as a
Deligne-Mumford stack and construct its universal family. We will first construct a moduli problem parametrizing non-compact stable toric varieties  and then compactify the universal family. The non-compact and the compact moduli problem coincide in codimension$1$. In section Section $3$ we will describe the inverse limits for GIT quotients corresponding to chambers in the image of the moment map that have a codimension 2 wall in common. In Section $4$ we will introduce our definition of virtual fundamental class.
In Section $5$ we will define tautological rings of inverse limits for certain subcategories of our category of GIT quotients and their flips. The relations between these rings could be viewed as analogous to the axioms satisfied by the virtual fundamental class in the case of the moduli space of stable maps. The significant result of this section is Proposition \ref{push and pull} which highlights the existence of pull-back and push-forward morphisms between these tautological rings. Also the Remark \ref{compute} points out an algorithm of computing the tautological rings. In section 6 we check some functorial properties of the virtual fundamental class, and the relation with the definition of a virtual fundamental class in \cite{behrend}. Section $7$ points out the advantages of the extending the constructions in this article to actions of reductive groups.

\textbf{Acknowledgements}
This article started from many long discussions with Valery Alexeev on the subject. Many of the constructions in this article are the fruits of these discussions.
I would like also to thank to Edwin O'Shea who taught me about shellability and the Bruggesser and Mani Theorem.

\section{Inverse limit of GIT quotients}

Let $X$ be a smooth projective variety with a torus $T= (\CC^*)^n$ action. We denote by $\mu:X \to \RR^n$ the moment map associated to a very ample line bundle $L$ on X. In this special case of  torus action, $\mu$ can be thought of as follows (see Chapter 8 in \cite{mum}). Consider an equivariant embedding $\iota :X \hookrightarrow \PP^m$ associated to a linearization of $L^k$. Implicitly, we regard $T$ as embedded in $(\CC^*)^m$, the maximal torus in the automorphism group of $\PP^m$. Let $\nu : \CC^{m+1} \to \RR^{m+1}$ be given by $$ \nu(x_0,...,x_m)= \left(\frac{|x_0|^2}{ \sum_{i=0}^m |x_i|^2},...,\frac{|x_m|^2}{\sum_{i=0}^m|x_i|^2}\right).$$ We will denote by $\{e_i\}_{i \in \{ 0,...,m\}}$ the standard basis in $\RR^{m+1}$.
After composition with the projection $\RR^{m+1} \to \RR^m\cong \RR^{m+1}/<\sum_{i=0}^n e_i>$, this induces a
morphism $\ol{\nu}:\PP^m \to \RR^m $. We can view $\RR^m$ as the cotangent space of $(S^1)^m$, and  $(S^1)^m$ as
naturally embedded in $(\CC^*)^m$, so that there is a natural projection $p:\RR^m \to \RR^n$ induced by the
embedding $(S^1)^n\cong T\cap  (S^1)^m \hookrightarrow (S^1)^m$. The moment map  is defined by
$\mu := p \circ \ol{\nu} \circ i$.


There exists a natural decomposition $\mu(X)=\bigcup_{l}F_l$ into a union of convex polyhedra, such that the complement of the union of the interiors of the top--dimensional polyhedra (the chambers of $\mu(X)$) are the image of the set $$\{x \in X \mbox{ such that } \mbox{dim}(\mbox{Stab}(x))>0\},$$ where $\mbox{Stab}(x)$ denotes the stabilizer of $x$. Each index $l$ corresponds to a GIT quotient  $ U_l// T $ of $X$ by $T$, where $ U_l=\begin{array}{ll}\{ x\in X; & p \in \mu (\overline{T\cdot x}) \} \end{array} $ is the locus of semi-stable points for any linearization given by a point inside the chamber $p\in F_l$. The quotient $ U_l// T $ is naturally isomorphic to the symplectic reduction $\mu^{-1}(p)/(S^1)^n\cong T\cdot \mu^{-1}(p)/T $.

Let $I$ be a set indexing all the chambers of $\mu(X)$. For each $i\in I$, we will denote by $M_i$ de GIT quotient $  U_i// T$, which in this case is an ordinary orbit space.

Let $J \subseteq \mbox{ Sym}^2 I$ be the set of codimension one faces $F_j$ which are boundaries of two chambers. For each $j=\{ i, i'\}\in J$,  we will denote by $M_j$ the space sitting over both $M_i$ and $M_{i'}$ in the flip from $M_i$ to $M_{i'}$.  For each $j\in J$,
 the boundary $ Y_j$  of the set  $\begin{array}{ll} \{ x\in X; &  \mu (\overline{T\cdot x})=F_j \}\end{array} \subset X $
is a nonsingular subvariety of $X$ whose generic points have a one-dimensional isotropy group $\CC^*_j$. The intersection $U_j \bigcap Y_j$ is the locus of semistable but not stable points for linearizations corresponding to interior points of $F_j$.


 Concretely, $F_j$ and $\CC^*_j$ can be understood as follows. Let $ H \subset \PP^m$ be the zero locus of a subset of the coordinates $\{x_0,...,x_m\}$. Consider the subtorus  \bea T_H := \begin{array}{lll} \{x \in (\CC^*)^m; & x\cdot y= y, & \forall y \in H \}\end{array}\eea of $ (\CC^*)^m$. Let $\cH$ be the set of all spaces $H$ as above such that $T_H \cap T$ is nontrivial and $\iota(X) \cap H \neq \emptyset$.  The walls of the chambers in the chamber decomposition of $\mu(X)$ are of the form $\ol{\nu}(\iota(X) \cap H)$ for some  $H\in \cH$. In particular, each codimension $1$ wall $F_j$ is the image of a space $H_j \in \cH$ such that $T_{H_j} \cap T=\CC^*_j$ is one dimensional.

 We will denote the quotient $T/\CC^*_j$ by $T_j$.


\begin{definition}\label{U_i/T} Let $ U_i$ be the stable--point locus for $X$ with a linearization corresponding to any point inside the chamber $F_i$. We will view $M_i$ as the coarse moduli space of the smooth Deligne-Mumford stack corresponding to the functor from schemes to sets $$ B \rightarrow \left\{ \diagram \cT \dto \rto^{\phi} & U_i \\ B \enddiagram  \right\},$$ where $\cT$ is a $T$--bundle over $B$ and $\phi$ is $T$--equivariant. We will denote  this Deligne-Mumford stack by $[X//T]_i$.
\end{definition}
\begin{definition}\label{U_j/T}
If $j=\{ i, i'\}$, let $U_j:=T\cdot \mu^{-1}((\mbox {int } F_i) \cup (\mbox {int } F_{i'}) \cup (\mbox {int } F_i \cap F_{i'}))$, where $\mbox {int }F$ denotes the interior of the set $F$.   We let $\cX//T_j$ be    the functor   from schemes to sets :
$$ B \rightarrow \left\{ \diagram \cT \dto^f\rto^{\phi} & U_j \\ B \enddiagram  \right\},$$ such that $\cT$ is quasiprojective over $B$, the morphism $f$ is flat, both $f$ and $\phi$ are $T$--equivariant (for the trivial action of $T$ on $B$), and for any closed point $b$ in $B$, the fiber $\cT_b$ of $f$ at $b$ satisfies one of the following conditions:

a) $\cT_b\cong T$, equivariantly with respect to the canonical $T$--actions, and $\mu(\phi(\cT_b))$ contains the interiors of both $F_i$ and $F_{i'}$; or

b) $\cT_b= T_{i} \cup T_{i'}$, where  \bea & T_{i}= \cT_b \bigcap \phi^{-1}\mu^{-1}((\mbox {int } F_i) \cup (\mbox {int } F_i \cap F_{i'}) ) \mbox{  and }& \\
 & T_{i'}= \cT_b \bigcap \phi^{-1}\mu^{-1}((\mbox {int } F_{i'}) \cup (\mbox {int } F_i \cap F_{i'}) ) & \eea
are irreducible varieties such that $T_{i'} \backslash T_i \cong T$, $T_i \backslash T_{i'} \cong T$ and $T_i \cap T_{i'} \cong T/ \CC_j^*=T_j$, equivariantly with respect to the canonical $T$--actions.

\end{definition}

Let $B_0\hookrightarrow B$ denote the locus of points over which the fibre is of the form b). We note that the  $f^{-1}(B_0)=\cT_{0i}\bigcup \cT_{0i'}$ where
\bea & \cT_{0i}= f^{-1}(B_0) \bigcap \phi^{-1}\mu^{-1}((\mbox {int } F_i) \cup (\mbox {int } F_i \cap F_{i'}) ) \mbox{  and }& \\
 & \cT_{0i'}= f^{-1}(B_0) \bigcap \phi^{-1}\mu^{-1}((\mbox {int } F_{i'}) \cup (\mbox {int } F_i \cap F_{i'}) ) & \eea
satisfy $[(\cT\setminus \cT_{0i})/\CC^*_j]\cong [(\cT\setminus \cT_{0i'})/\CC^*_j]\cong$ the categorical quotient $\cT/\CC^*_j$.

\bigskip

\begin{notation} As before, we consider the equivariant embedding  $\iota : X \hookrightarrow \PP^m$ corresponding to a linearization on a very ample line bundle $L^k$ on $X$. Let $d$ be the class of the closure of the generic orbit of $\CC^*_j$ in $H_2(\PP^m)$ and $\mu_j:\PP^m \to \RR$ the moment map associated to
 $\CC^*_j$. For $x \in \PP^m$ generic and $t\in \CC^*_j$ we will denote $$t_0 = \mu_j(\lim_{t \to
 0}tx) \mbox{ and } t_{\infty} = \mu_j ( \lim_{t \to \infty} tx).$$
 Let
  $\{t_s\}_{s }$ be the walls of $\mu_j(\PP^m)$ and $$Y_{j,s}= \{ x \in X \mbox{ such that } \mu_j(\CC^*_j x
 )=t_s \}.$$ We denote by $k_s$ the degree of $\ol{\CC^*_j x}$ for
 all $x \in \PP^m$ such that $\mu_j(\ol{\CC^*_j x}) = [t_0, t_s]$.
 Note that all such $x$ form an open set in an equivariant linear
 subspace in $\PP^m$.


We will denote by   $\ol{M}_{0, \cA_s}(\PP^n, d, a)$  the
Deligne-Mumford stack of weighted stable maps of class $d$ with 2
marked points, where the weights on the marked points are given by
the pair $\cA_s =( \frac{d-k_s}{d} + \epsilon, \frac{k_s}{d} +
\epsilon)$, for $\epsilon
>0$ is very small, and the weight on the map is $a =\frac{1}{d}$ (see
\cite{noi2}). We will denote by $\ol{M}_{0, \cA_s}(X, d, a)$  the
substack of $\ol{M}_{0, \cA_s}(\PP^n, d, a)$ given by those maps whose
image is contained in $X$. Similarly we define $\ol{M}_{0, \cA_{s,s'}}(X,d,a)$ where
$t_{s'} >t_s$ are two consecutive walls and $\cA_{s,s'} = ( \frac{d-k_s}{d}, \frac{k_{s'}}{d})$.
\end{notation}

\begin{proposition}\label{moduli}
a) If $(t_s, t_{s'}) \subset \mu_j(\iota(\CC^*_j x))$  for all $x\in U_i$, then there exists a natural imbedding $[X//T]_i \subset [\ol{M}_{0, \cA_{s,s'}}(X,d,a)/T_j]$.

b)There is a smooth Deligne-Mumford stack $[X//T]_j$ which represents the functor $\cX//T_j$, and  $M_j$ is the coarse moduli space of $\cX//T_j$. Moreover there is a natural embedding $[X//T]_j \subset [\ol{M}_{0, \cA_s}(X,d,a)/T_j]$.

\end{proposition}

\begin{proof}

 Let $ \left\{ \diagram \cT \dto^f\rto^{\phi} & U_j \\ B \enddiagram \right\} $ be a set of data as in the Definition \ref{U_j/T} above (respectively, Definition \ref{U_i/T}).
 We will first prove the existence of a morphism
 \bea  \cT/\CC^*_j \to \ol{M}_{0, \cA}(X,d,a), \eea
for the set of weights $(\cA,a):= (\cA_s,a)=( (\frac{d-k_s}{d} + \epsilon, \frac{k_s}{d} + \epsilon), \frac{1}{d})$ (respectively $(\cA,a):=(\cA_{s ,s'},a)=( (\frac{d-k_s}{d} , \frac{k_{s'}}{d} ), \frac{1}{d})$).

We recall that such a morphism is given by a family $$\left\{ \diagram \cC \dto^{}\rdashed|>\tip & X \\ \cT/\CC^*_j \enddiagram  \right\}$$
of weighted stable maps with weights $((a_0, a_{\infty}), a)$ is determined by a line bundle $\cL_{\cC}$ on $C$, of fixed degree $d$ on each fiber $C_p$, for any closed point $ p \in \cT/\CC^*_j$, and a morphism
$e:\cO_C^{n+1}\to L_{\cC}$ (specified up to isomorphisms of the target)
satisfying a series of stability conditions:
\begin{enumerate}
\item $\omega_{C|S}(a_0s_0+a_{\infty}s_{\infty})\otimes L_{\cC}^a$ is relatively ample over $ S=\cT/\CC^*_j$,
\item $\Coker e$, restricted over each fiber $C_p$, is a skyscraper sheaf supported only
on smooth points of $C_p$, and
\item for any $ q \in \cC_p$ and for any $ i \in \{ 0, \infty \} $ (possibly empty) such that
$ q= s_i(p)$, the following holds:  $$ a_i + a\dim\Coker e_{q}\leq 1.$$
\end{enumerate}

 \bigskip

We first construct $\cC$ as above. Let $\cT^j$ be the union of all $\CC^*_j$--fixed point loci
  of $\cT$.
  (When working with Definition \ref{U_i/T}, then $\cT^j$ is the empty set.)
  Consider the action of  $(\CC^*)^2$ on $\cT \times (\CC^2\setminus\{ (0,0) \})$ defined as follows: on the first factor, via the multiplication map $(\CC^*)^2 \to \CC^* \simeq \CC^*_j$ composed with the induced action on $T$; on the second factor, as
  $(u,v) * (x,y) = (ux, v^{-1} y)$.
The locus of points with non-trivial stabilizers in $\cT \times (\CC^2\setminus\{ (0,0) \})$ is $\cT^j\times (xy=0)$.
Thus the quotient
\bea  \cC := (( \cT \times (\CC^2\setminus\{ (0,0) \}) \setminus (\cT^j\times (xy=0)) )/  (\CC^*)^2 \eea
is an orbit space. The embedding $i: \cT \to \cT \times (\CC^2\setminus\{ (0,0) \}) $ given by $i(p):=(p, (1,1))$ leads to an embedding $\cT\hookrightarrow \cC$, while the projection on the first factor
$p_1: \cT \times (\CC^2\setminus\{ (0,0) \}) \to \cT$ induces a flat morphism $\cC \to \cT/\CC^*_j$.
 Moreover,  we claim that $\phi$ induces  a family of weighted stable maps to $X$
\bea  \cC \dashrightarrow \PP^n \eea
 for the appropriate weights $(\cA_s,a)=( (\frac{d-k_s}{d} + \epsilon, \frac{k_s}{d} + \epsilon), \frac{1}{d})$ (respectively $(\cA_{s ,s'},a)=( (\frac{d-k_s}{d} , \frac{k_{s'}}{d} ), \frac{1}{d})$)
  where the two marked sections are given by the complement of $\cT$ in $\cC$. We will call these marked section $s_0$ and $s_{\infty}$ where $ \lim_{t \to 0} t q \in s_0(\cT/\CC^*_j)$ and $\lim _{t \to \infty} t q  \in s_{\infty}(\cT/\CC^*_j)$ for $q \in \cT$ generic and $t\in \CC^*_j$.

   We proceed now to construct $L_{\cC}$ and its $(n+1)$ global sections which will define $e$. Let $\cG$ denote the graph of the rational map from $\cC$ to $X$ defined by $\phi$, let $\ol{\cG}$ denote its closure in $\cC\times X$, and let $\pi_1: \ol{\cG} \to \cC$ and $\pi_2 : \ol{\cG} \to X$ be the two natural projections. Consider any non-zero section $\sigma$ in $L^k$, such that its intersection with the zero section defines a scheme $ h $ with the property that $\pi_2^{-1}(h)$ doesn't contain any generic point of $\pi_1^{-1}(s_0(\cT/\CC^*_j))$ or $\pi_1^{-1}(s_{\infty}(\cT/\CC^*_j))$.
  We claim that $\pi_1(\pi_2^{-1}(h))$, which a priori is only a Weil divisor, is in fact a Cartier divisor and
  \bean \label{L_C} L_{\cC} = L(\pi_1(\pi_2^{-1}(h))) \bigotimes L(s_0(\cT/\CC^*_j))^{k_{s_{0}}} \bigotimes L(s_{\infty}(\cT/\CC^*_j))^{d-k_{s_{\infty}}}\eean for some $k_{s_{0}}$ and $k_{s_{\infty}}$ chosen such that $L_{\cC} $ will fulfill the requirements of a family of stable maps as above.

We will now check that the above Weil divisor is a Cartier divisor.
Equivalently, we extend the line bundle $\phi^*L^k$ and its section $\sigma$ to $\cC$. However it is more convenient to work on the smooth covers of $\cT$  and $\cC$, respectively:
\bea  \diagram  {\cT  \times (\CC^*)^2 }  \dto^{\pi} \rrto && \cT\times (\CC^2 \setminus \{(0,0)\}) \setminus (\cT^j\times (xy=0)) \dto \\
\cT \rrto && \cC.  \enddiagram \eea
  So we let $m : \cT  \times (\CC^*)^2 \to \cT  \times (\CC^*)^2$ be the isomorphism given  on the first
  factor by the action of   $(\CC^*)^2$ on $\cT$ described above, and on the second factor by the projection
   $p_2:\cT  \times (\CC^*)^2 \to  (\CC^*)^2$. Let $p_1: \cT  \times (\CC^*)^2 \to \cT$ denote the projection on the first factor. Thus $m^{-1}\circ p_1$ is the quotient morphism $\pi$.
  The line bundle  $L'':=(m^{-1})^*p_1^*\phi^*L^k$
 can be extended to a line bundle $L'$ on $\cT \times (\CC^2 \setminus \{ (0,0)\}) $ which locally on affine connected
 open subsets $\Spec (R) \subset \cT$ is defined as follows.  For any section in $\Gamma( \Spec(R) \times (\CC^*)^2, L'')$
  of coordinate $f'' \in R[x, x^{-1}, y , y^{-1}]$, a corresponding section in $\Gamma(\Spec(R)\times (\CC^2 \setminus \{(0,0)\}), L')$ is given by $f' = f'' x^a y^b$, where
  $a$ and $b$ are minimal integers such that $f'\in R[x,y]$.

  The
  line bundle $L'_{|\cT\times (\CC^2 \setminus \{(0,0)\}) \setminus (\cT^j\times (xy=0))}$ with its given global section descends to
$L(\pi_1(\pi_2^{-1}(h)))$ on $\cC$, while the trivial line bundle on
$\cT\times (\CC^2 \setminus \{(0,0)\})$, with the linearizations
induced by the action of $(\CC^*)^2$ on $x$ and $y$, respectively,
induce $L(s_0(\cT/\CC^*_j))$ and $L(s_{\infty}(\cT/\CC^*_j))$ on
$\cC$.
With the notations set up in the preamble to the proposition, let
$t_{s_{0}}$ and $t_{s_{\infty}}$
  denote the limits of $\mu_j(tq)$ when $t\to 0$ and $t \to \infty$, respectively, where $q$ is the generic point in $\Spec (R)$ and $t\in \CC^*_j$.

A suitable choice of $(n+1)$ linearly independent global sections on
$L^k$ yields $(n+1)$ global sections on $L(\pi_1(\pi_2^{-1}(h)))$,
and thus a morphism $e_1: \cO_{\cC}^{n+1}  \to
L(\pi_1(\pi_2^{-1}(h)))$.  Then the support of $\Coker e_1$ is
clearly included in the complement of $\cT$ in $\cC$. The choice of
the powers $k_{s_{0}}$  and $d-k_{s_{\infty}}$ in formula \ref{L_C}
  comes from inspecting the rank of $\Coker e_1$ on the fiber of $\cC$ over general closed points $p \in \cT/ \CC^*_j$.

 a) For this we will first investigate  the case when $ \left\{ \diagram \cT \dto^f\rto^{\phi} & U_j \\ B \enddiagram \right\} $ is a set of data as in the  Definition \ref{U_i/T}. In this case we have the advantage of knowing that $[X//T]_i$ is
  represented by a smooth Deligne-Mumford stack. Consider $B'$ an \'etale open cover over this stack.
Thus $B'$ comes with a smooth family $\cT'$ mapped to $X$. A family $\cC'$ over $\cT'/\CC^*_j$, a rational map from $\cC'$ to $X$, and the closure $\ol{\cG}'$  of its graph, with projections $\pi_1':\ol{\cG}'\to \cC'$ and $\pi_2':\ol{\cG}'\to X$  are associated to $\cT' \to B'$ as above.
As we are interested only in the fibers over the closed points $p \in \cT/\CC^*_j$, we can further restrict our study to a smooth curve in $\cT'/\CC^*_j$ passing through the point $p'$ such that there is an isomorphism   $\cT'_{p'} \cong \cT_p$ of schemes over $X$, and such that the fiber of $\cT'$ over
  a general point in this curve is mapped to a curve of degree $d$ in $X\hookrightarrow \PP^n$.
  Let $\ol{\cG}'_{p'}$ denote the preimage of $p'$ in $\ol{\cG}'$.
   Then we can write $\ol{\cG}_{p'}= C_0 \cup C_1 \cup C_{\infty}$ such that $\pi'_1(C_0) = s'_0(\cT'/\CC^*_j)$ and
   $\pi'_1(C_{\infty})= s'_{\infty}(\cT'/\CC^*_j)$
   and $\pi_1'(C_1) = \cC'_{p'}\cong \cC_p$. Moreover our choices
   imply
   \bea  \deg(\pi'^*_2(L^k)_{|C_0})+\deg(\pi'^*_2(L^k)_{|C_1})+\deg(\pi'^*_2(L^k)_{|C_{\infty}}) =d. \eea
In this case we can define $L_{\cC'}:=L(\pi'_1(\pi'^{-1}_2(h)))$, coming as before with a choice of global sections which induce   $e': \cO_{\cC}^{n+1}  \to
L(\pi'_1(\pi'^{-1}_2(h)))$.

   Because $(t_s, t_{s'}) \subset \mu_j(\iota(\CC^*_j x))$  for all $x\in U_i$,  we deduce that
   \bea  k_{s_{0}} := \dim \Coker e'_{s'_0(p')}= \deg(\pi'^*_2(L^k)_{|C_0}) \leq
   k_s, \mbox{ and } \\
     k_{s_{\infty}}:=d- \dim \Coker e'_{s'_{\infty}(p')} =d-\deg(\pi'^*_2(L^k)_{|C_{\infty}}) \geq k_{s'},   \eea
   where $k_s$, $k_{s'}$, $k_{s_{0}} $ and $k_{s_{\infty}}$ are like in the notations preceding this Proposition.
   Interpreted in local coordinates
   this justifies the definition of the line bundle $L_{\cC}$  given above by equation (\ref{L_C}). Indeed, by tensoring $e_1$ with the chosen powers of two fixed sections in $L(s_0(\cT/\CC^*_j))$ and $L(s_{\infty}(\cT/\CC^*_j))$ (respectively), we obtain $e: \cO_{\cC}^{n+1}  \to L_{\cC}$ satisfying the required stability conditions.

 The induced morphism $\cT/\CC^*_j \to \ol{M}_{0, \cA_{s,s'}}(X, d, a)$ is $T_j$--equivariant. Also, $T_j$ acts freely on $\cT/\CC^*_j$, with quotient $B$.
  In conclusion there exists a morphism $ B \to [\ol{M}_{0, \cA_{s,s'}}(X, d, a)/T_j]$ such that $\cC$ is the pull-back of the universal family over the stack $\ol{M}_{0, \cA_{s,s'}}(X, d, a)$. In this way, one constructs a morphism from
   $[ \cX //T ]_i$ to $[\ol{M}_{0, \cA_{s,s'}}(X, d, a)/T_j]$ whose image is $M/T_j$, where $M$ is the open set of points with finite stabilizer in a fixed loci of $\ol{M}_{0, \cA_{s,s'}}(X, d, a)$ for the action of $\CC^*_j$. The construction of the inverse $ M/T_j\to [\cX //T]_i $ is a straightforward exercise at the level of universal families.

 Case b).   For $ \left\{ \diagram \cT \dto^f\rto^{\phi} & U_j \\ B \enddiagram \right\} $ a set of data as in the  Definition \ref{U_j/T}, with $j=\{i, i'\}$, it is enough to notice (with the notations in  Definition \ref{U_j/T}), that
 the families
  $\cT\setminus \cT_{0i'}\to B$ and $\cT\setminus \cT_{0i}\to B$ determine morphisms
   $\phi_i: B \to [X//T]_i$ and $\phi_{i'}: B \to [X//T]_{i'}$. Indeed, then  $\cC\setminus \ol{\cT_{0i'}}$ and $\cC\setminus \ol{\cT_{0i}}$, respectively,
   are open subsets in the pull-backs of the universal families on $\cT/\CC^*_j \to \ol{M}_{0, \cA_{s,s'}}(X, d, a)$.
   Thus they come with a line bundle and $(n+1)$ global sections pulled back from these universal families.
   In this case the line bundle $\cL_{\cC}$ and the morphism $e: \cO_{\cC}^{n+1}  \to L_{\cC}$ can be constructed by gluing the line bundles with sections on $\cC\setminus \ol{\cT_{0i'}}$ and $\cC\setminus \ol{\cT_{0i}}$ with the bundle $\phi^*L^k$  on the open $\cT \subset \cC$, and its chosen set of global sections.





 As before, the induced morphism $\cT/\CC^*_j \to \ol{M}_{0, \cA_s}(X, d, a)$ is $T_j$--equivariant, and induces a morphism $ B \to [\ol{M}_{0, \cA_s}(X, d, a)/T_j]$ such that $\cC$ is the pull-back of the universal family over $\ol{M}_{0, \cA_s}(X, d, a)$. As before, the induced morphism
   \bea  \cX //T_j \to [\ol{M}_{0, \cA_s}(X, d, a)/T_j] \eea is an embedding.

  This proves the existence of the  Deligne-Mumford stack $[X//T]_j$, with universal family $\cU_j$.

\bigskip

  We will proceed with an explicit construction of $[X//T]_j$ in local coordinates.

   We first construct the universal family over $[X//T]_j$. There exists a Zariski open cover of $U_j$ by $T$--equivariant affine open varieties. Let $U = \Spec R$ be such a variety and let $R= \bigoplus_{n \in \ZZ} R_n$ be the grading of $R$ with respect to the weights of the induced $\CC^*_j$--action. Let $\{ x_{i,k}  \}$, with $x_{i,k} \in R_i$,  be a minimal set of generators of the ideal $I$ generated by $ \bigoplus_{i \neq 0} R_i$. Notice that  $\Spec (R/I)=: Y_j$ is smooth, being the fixed locus of $U$  for the
   $\CC^*_j$--action. We can take $U$ small enough for $\{ x_{i,k} \}$ to be a regular sequence, i.e. such that the image of $\{ x_{i,k}\} $ in $I/I^2$ is a basis of the conormal bundle of $Y_j$ in $U$.
   Let $R'=R[\{ y_{i,k} \}]/(y_{i,k}^{|i|}-x_{i,k})_{i,k}$ and $U' =\Spec R'$. Thus $U'/\Lambda \cong U$, where $\Lambda = \prod_{i,k} (\ZZ / i \ZZ)^k$, and the morphism from $U'$ to $U$ is flat. Moreover, there is an action of $\CC^*_j$ on $U'$ with weights $\pm 1$ such that the natural morphism from $U'$ to $U$ is $\CC^*_j$--equivariant. Let $I^+, I^- \subseteq I$ be the ideals generated by those $x_{i,k}$ with $i < 0$, and $i>0$, respectively.  Let $BY_j^{\pm} := \Spec (R/I^{\pm})$.  We will denote by $BY'^{\pm}_j$ the pullback of $BY_j^{\pm}$ to $U'$, and by $\Bl_{BY'^{\pm}_j} U'$ the blow-up of $U'$ along $BY'^{\pm}_j$. Let $\Bl (U') := \Bl_{BY'^+_j} U' \times_{U'} \Bl_{BY'^-_j} U'$. Let $p^{\pm}: BY'^{\pm}_j \to Y'_j$ denote the projections of the  two Bialynicki-Birula strata to the fixed loci  and let $N_{Y'_j|BY'^{\pm}_j}$ denote the normal bundle of $Y'_j$ in $BY'^{\pm}_j$. As all the weights of the $\CC^*_j$ action on $U'$ are $\pm 1$, there are natural isomorphisms $BY'^{\pm}_j \cong N_{Y'_j|BY'^{\pm}_j}$ and  $p^{\pm *}(N_{Y'_j|BY'^{\mp}_j}) \cong N_{BY'^{\pm}_j| U'}$     (as shown in \cite{bs}). Thus any closed point in $\Bl (U')$ can be uniquely associated to either
   \begin{itemize}
   \item[i)] a point $x$ in $U'$ whose $\CC^*_j$--orbit is closed in $U'$; or
   \item[ii)] a point $x$ in $BY'^{\pm}_j$, together with an orbit in $BY'^{\mp}_j$ whose closure intersects $\overline{\CC^*_j\cdot x}$ at a point in $Y'_j$.
   \item[iii)] a point $y$ in $Y_j$ together with two orbits, one in each of the strata $BY'^{\mp}_j$, whose closures pass through $y$.
   \end{itemize}
     Let $\cU'$ denote the set of all objects described by i)--iii), and $\cU$ de set of their $\Lambda$--classes. The  action of $\Lambda$ on $U'$ induces a natural action on $\Bl (U')$, and  $\Bl (U')/\Lambda$ is in bijective correspondence with $\cU$. We claim that the coarse moduli space for the universal family over $[X//T]_j$ is locally isomorphic to $\Bl (U')/\Lambda$. This can be justified by a sequence of standard arguments. For each point $x \in \Bl(U')$ we consider the largest subgroup $\Lambda_x \subseteq \mbox{ Stab}_x\subseteq  \Lambda$ of the stabilizer of $x$, which acts trivially on the object in $\cU'$ parametrized by $x$. We notice that $\Lambda_x$ is a small subgroup of $\mbox{ Stab}_x$ and, unless all the weights of $\CC^*_j$ on $I^+$ or $I^-$ are the same, $\Lambda_x$ is the largest small subgroup of $\mbox{ Stab}_x$. We choose an open subset  $V'_x \subseteq  \Bl(U')    $ containing $x$ and on which $\Lambda_x$ acts like a small group.  By "The purity of the branch locus" Theorem (see \cite{vistoli}), there is a unique smooth stack structure having  $\{ V'_x/\Lambda_x \}_{x\in U} $ as an \'etale atlas and $\Bl (U')/\Lambda$ as coarse moduli space. We will call this stack $[U_j]$. It comes with natural morphisms $\Bl (U') \to [U_j]$ and $[U_j] \to U'/\Lambda\cong U \subset U_j$. We claim that $[U_j]$ is an open substack of the universal family over $[X//T]_j$.

   All the GIT quotients and the categorical quotient of $\Bl (U')$ with respect to $\CC^*_j$ are isomorphic and will be denoted by $\Bl (U')/\CC^*_j$. There is a natural morphism $[U_j] \to (\Bl(U')/\CC^*_j)/\Lambda$. Again there is a Deligne-Mumford stack $D$ having $(\Bl(U')/\CC^*_j)/\Lambda$ as a coarse moduli space, which can be constructed using "The purity of the branch locus" Theorem  as above. Moreover, the morphism $[U_j] \to (\Bl(U')/\CC^*_j)/\Lambda$ factors through a morphism $g: [U_j] \to D$. To check that the composition
   $[U_j] \to D \to [D/T_j]$ is flat reduces to the proof that $\Bl(U') \to \Bl(U')/\CC^*_j$ is flat which follows by direct computation. The morphism  $[U_j] \to [D/T_j]$  together with the $T$-equivariant morphism $[U_j] \to U_j$, satisfy the conditions in Definition \ref{U_j/T}, and thus determine
   a morphism  $g_j:[D/T_j] \to [X//T]_j$  such that
   $[U_j] \to [D/T_j]$ is the pullback through $g_j$ of the universal family $\cU_j \to [X//T]_j$. It remains to argue that $g_j$ is an open embedding.
   As a first observation, the bijection $\Bl(U')\to \cU$ above implies that $g_j$ induces an injective map
   \bea  \Hom(\Spec\CC; [D/T_j]) \to \Hom(\Spec\CC; [X//T]_j) \eea
  at the level of points and an isomorphism $\Spec\CC\times_{[D/T_j]} \Spec\CC \cong \Spec\CC \times_{[X//T]_j} \Spec\CC$.
    Finally  we check that $g_j$ induces isomorphisms on tangent spaces. Over points parameterizing irreducible orbits this is obvious. On the other hand,
    the pull-back of the boundary divisor in $\ol{M}_{0, \cA}(X, d, a)/T_j$   to $[D/T_j]$ (or $ [X//T]_j$, respectively) consist exactly in the substack of $[D/T_j]$ (or $ [X//T]_j$, respectively) parameterizing reducible varieties. The morphism between these substacks obtained as a restriction of $g_j$ is clearly an open embedding, due to their respective fibred product structures. Together, the above show that $g_j$ is an open embedding.


\end{proof}

\begin{definition}\label{graph}
We define the graph $G$ associated to a smooth projective variety $X$ with a torus $T= (\CC^*)^n$ action as follows:
\begin{itemize}
\item The set of vertices is $K=I \cup J$, where   $I$ is the set of all the chambers of $\mu(X)$, and  $J \subseteq \mbox{ Sym}^2 I$ is the set of codimension one faces $F_j$ which are boundaries of two distinct chambers.
 \item The edges are of the form $(j, i)$ where $i\in I$, and $j\in J$ is given by a pair $\{i,i'\}\subset I$ with $i\not=i'$.
\end{itemize}

\end{definition}

\begin{remark}\label{colimit}
For any subgraph $ \Gamma \subseteq G$, we denote by $ \underleftarrow{\lim}_{\Gamma}[X//T]$ the colimit of all spaces $[X//T]_k$ corresponding to nodes $k$ of $\Gamma$. We will denote by $I_{\Gamma}$ and $J_{\Gamma}$ the subsets of $I$ and $J$ such that $\{M_{i}\}_{i \in I_{\Gamma}}$ and $\{M_{j}\}_{j \in J_{\Gamma}}$ correspond to the nodes of $\Gamma$, and by \bea e_{ji}:[X//T]_j \to [X//T]_i\eea  the morphisms corresponding to  the set $E(\Gamma)$ of edges in $\Gamma$. Then $\underleftarrow{\lim}_{\Gamma}[X//T]$ can be constructed as the fiber product in the Cartesian diagram
         \bea \diagram  {\underleftarrow{\lim}_{\Gamma}[X//T]}  \rrto \dto && {\prod_{j \in J_{\Gamma}} [X//T]_j} \dto_{m} \\
         {\prod_{i \in I_{\Gamma}}[X//T]_i} \rrto^{\Delta} && {\prod_{e_{ji}\in E(\Gamma)}[X//T]_i}. \enddiagram \eea
 Here $\Delta$ is the diagonal morphism and  $m=\prod_{e_{ji}\in E(\Gamma)} e_{ji}\pi_j$ where $\pi_j$ is the projection on the $j$-th factor of $\prod_{j' \in J_{\Gamma}} [X//T]_{j'}$.

 When $\Gamma$ is a tree, this construction leads to a  class $v(\underleftarrow{\lim}_{\Gamma}[X//T])$ of dimension $\dim(X) - \dim(T)$ in the Chow group of  $\underleftarrow{\lim}_{\Gamma}[X//T] $: $$ v(\underleftarrow{\lim}_{ \Gamma}[X//T]) := \Delta^![\prod_{j \in J_{\Gamma}} [X//T]_j].$$
We will call this the "virtual" class of $\underleftarrow{\lim}_{\Gamma}[X//T]$.
\end{remark}

Before describing the stack structure of $\underleftarrow{\lim}_{\Gamma}[X//T]$, we introduce notations for the possible structures of objects it parametrizes.

\begin{definition} Let $\Gamma$ be any subgraph of $G$, and let $J'\subseteq J$. A graph $\Gamma'$ will be called a $J'$--contraction of $\Gamma$ if it can be constructed from $\Gamma$ by deleting the edges $e_{ij}$ and $e_{ji'}$, and identifying the $i$-th, $j$-th and $i'$-th node for the elements $j=\{ i, i'\} \in J'\subseteq J_{\Gamma }$,  leaving all other nodes and edges unchanged. The nodes of $\Gamma'$ will still be partitioned accordingly into sets $I_{\Gamma'}$ and $J_{\Gamma'}$, with the only extra mention that each triple of nodes $i,i',j$ identified by the contraction is regarded as an element in $I_{\Gamma'}$.

A $T$--equivariant morphism  $\varphi: T_{\Gamma'} \to X$  will be called a $\Gamma'$--type map if its domain is a union $T_{\Gamma'}= \cup_{i \in I_{{\Gamma'}}} T_i$, where $T_i$ are irreducible varieties and
\begin{enumerate}
\item $T_i \cap T_{i'} = \emptyset$ if $i \neq i'$ and $j=\{ i, i'\}$ is not an element of $J_{{\Gamma'}}$;
\item If $\{ i, i'\}=j \in J_{{\Gamma'}}$, then   $T_i \cap T_{i'}  \simeq T/ \CC_j^*$ equivariantly.
\item For each $i\in I_{\Gamma'}$,  we have $T_i\setminus (\cup_{i' \in I_{\Gamma'}\setminus\{ i\}} T_{i'})\simeq T$ equivariantly,
and it is mapped by $\mu\circ\varphi$ into
 $\bigcup_{l\in [i]}(\mbox{int } F_i),$ where $[i]$ denotes the class of all elements in $K_{\Gamma}$ identified with $i$ by the contraction of $\Gamma$ to $\Gamma'$.

\end{enumerate}
\end{definition}

\begin{proposition} (Gluing by fiber products)

For any subgraph ${\Gamma} $ of $G$, let $U_{\Gamma}:= \cup_{i \in I_{{\Gamma}} } U_i \bigcup \cup_{j \in J_{{\Gamma}} }Y_j$, where
  $ Y_j$  denotes the closure of the set  $\begin{array}{ll} \{ x\in X; &  \mu (\overline{T\cdot x})=F_j \}\end{array} \subset X$ for the wall $F_j$ corresponding to $j\in J$.

The colimit $\underleftarrow{\lim}_{\Gamma}[X//T]$ is a Deligne-Mumford stack corresponding to the functor
 $$ B \rightarrow \left\{ \diagram \cT \dto^f\rto^{\phi} & U_{\Gamma} \\ B \enddiagram  \right\},$$
  such that $\cT$ is quasiprojective over $B$, the morphism $f$ is flat, $f$ and $\phi$ are $T$--equivariant morphisms and the restriction of $\phi$ to any fiber of $f$ is a  ${\Gamma'}$--type morphism for some contraction $\Gamma'$ of $\Gamma$ as above.
\end{proposition}

\begin{proof}  Consider the morphisms $q_j:\underleftarrow{\lim}_{\Gamma}[X//T]\to [X//T]_j$. The universal family over $\underleftarrow{\lim}_{\Gamma}[X//T]$  can be obtained by gluing together the pullbacks
 of the universal families over all $[X//T]_j$ for all  $j\in J_{\Gamma}$. Indeed, consider $j=\{i,i'\}\in J_{\Gamma}$ and the corresponding morphism $e_{ji}:[X//T]_j \to [X//T]_i$, and let  $p_j: \cU_j\to [X//T]_j$ and $p_i: \cU_i \to [X//T]_i$, denote the universal families over the domain and target, with their morphisms $\phi_j: \cU_j \to X$ and  $\phi_i: \cU_i\to X$. As noted in the proof of Proposition \ref{moduli}, there is a Cartier divisor $D_j$ in $[X//T]_j$ such that $p_j^{-1}(D_j)$ consists exactly of the reducible fibers of $p_j$. The complement $\cC_{ij}$ of $p_j^{-1}(D_j) \bigcap \phi_j^{-1}\mu^{-1}(\mbox{ int } F_{i'})$ in $\cU_j$ is a $T$--bundle over $[X//T]_j$ which determines the map $e_{ji}$. The pullbacks of $q_j^*\cU_j$ can now be glued  along $q_j^*\cC_{ij}\cong q_{j'}^*\cC_{ij'}$ to yield a family $\cU_{\Gamma}$ with the required properties. For any family $\cT\to B$,  $\phi:\cT\to U_{\Gamma}$ with these properties, one finds open subschemes $\cT_j$ and $\cT_i$ of $\cT$ which induce compatible morphisms $B\to [X//T]_j$ and $B\to [X//T]_i$ respectively, and thus a map $\rho: B\to \underleftarrow{\lim}_{\Gamma}[X//T]$ such that  $\rho^*\cU_{\Gamma}\cong \cT$.

\end{proof}

In the case $\Gamma= G$ we proceed now in compactifying the fibers of any family $f:\cT \to B$. There is a natural way to comactify each  fiber of $f$. For each $ x \in B$ and each orbit $T_{i,x}$ of the action of $T$ on $\cT_x$ there is a natural compactification of $T_{i,x}$ which is compatible with its comactification in $\cT$.  For each boundary  divisor $T_j$ of $T_i$ there is one parameter group $i_j:\CC^*_j \hookrightarrow T$ such that $\lim_{t \rightarrow 0} i_j(t)x \in T_j$ for a generic $x$ in $T_i$. We Denote by $\AA^1_j$ the affine line obtained by adding the point $0$ to $\CC^*_j$. The compactification of $T_i$ could be obtained as a GIT quotient of $\prod_{i \in j}\AA_j^1 \times T_{i,x}$ via the natural action of $\prod_{i \in j} \CC^*_j$ (see \cite{cox}). It would be then easy to see that these compactifications can be glued along the boundaries to give a compactification of the fiber $\cT_x$. However we have to see that these fibers fit into a flat family over $B$ so we would like to compactify all $\cT$.

 The line bundle $\cL=\phi^*(\cO (1))$ pulled-back from $X$ to$\cT$ is ample along the fibers of $\cT \to B$ as $\phi$ restricted to anyone of these fibers is ample. For any fiber $\cT_x$ of $f$ we pick  a number $k$ large enough such that $\cL_{|\cT_x}^k$ is very ample and consider the map $h_x:\cT_x \to \PP^N$ given by the equivariant sections of $\cL^k$. We also notice that this map can be extended on $\cT_U$ for some etale neighborhood $U$ of $x$ such that $h_y:\cT_y \to \PP^N$ is an embedding for any $y \in U$. We can compactify the fibers of $\cT_U $ by taking their closure $\overline{\cT}$ in  $U \times \PP^n$.

\begin{proposition}

For any $x \in U$ the fiber of the projection $\overline{f}: \overline{\cT_U} \to U$  restricted to $x$ is $\overline{\cT_x}$ and the family $\overline{f}: \overline{\cT_U} \to U$ is flat.

\end{proposition}

\begin{proof} We will first need to check that the Hilbert polynomial of $\overline{\cT_x}$ is independent of $x$. We will show that $h^0(L^k,\overline{\cT_x})$ is independent of $x$ and the same is true for any  power of the line bundle. Let $\mu':\cT_U \to \RR^k$ be the moment map associated to the  linearization of $\cL^k$. There is a natural lattice inside $\RR^n$ corresponding to the 1 forms pulled back via equivariant maps from $(\CC^*)^n \to \CC^*$. For each orbit $T_{i,x}$ the equivariant global sections of $\cL^k$ on $\overline{T_{i,x}}$ are in bijections with the lattice points in $ \mu'(\overline{T_{i,x}})$. The global sections which are in the interior of $\mu'(T_{i,x})$ extend by $0$ on all the other components of $\overline{f}_x$. As  a consequence, for any $x \in U$  $h^0(L^k,\overline{\cT_x})$ is the number of lattice points in $\mu'(\overline(\cT_U))$. The same proof works for any power of $\cL^k$. So far this shows that the pullback of  $\overline{\cT_x}$ to the reduced scheme $U_{red}$ is flat. Let $\cN$ be the nilpotent ideal of $U$. There exists an integer $l$ such that $\cN^l=0$.

Using the exact sequences $$ 0 \rightarrow \cN^s \otimes f_*(\cL^k ) \rightarrow  \cN^{s-1} \otimes f_*(\cL^k ) \rightarrow  \frac{ \cN^s}{  \cN^{s-1}} \otimes  f_*(\cL^k ),$$ for all $0<s<l$ we obtain that $ f_*(\cL^k ) $ is finitely generated.
If $k$ is large enough $R^1f_*(\cL^k )=0$  so using a Cech resolution of  $ f_*(\cL^k ) $ and the fact that the morphism from $\cT_U \to U$ is flat we obtain that  $ f_*(\cL^k ) $ is locally free.
Notice that this implies in particular that for large enough $k$,   $ f_*(\cL^k ) $ is the same with $\overline{f}_*(\cL^k)$. As we can recover $ \overline{\cT_U}$ from the sum of  $\overline{f}_*(\cL^k)$ for high enough powers of $k$ we conclude that   $ \overline{\cT_U}$ is flat over $U$.

\end{proof}

 To compactify the whole family $f :\cT \to \PP^n$ we embed it in $Proj(f_*(L^k))$ for a large enough $k$ and consider its closure $\overline{\cT}$. It is easy now to check that the morphism $\phi$ extends to $\overline{\cT}$. We conclude that for $ X= \PP^n$,  $\underleftarrow{\lim}_{\Gamma}[X//T]$  is the same moduli space as the one constructed in \cite{alex} Theorem 2.10.10..

In the special case $T= \CC^*$ $\underleftarrow{\lim}_{\Gamma}[X//T]$ is embedded in $\ol{M}_{0, 0}(X, d)$ as a union of some of the fixed loci of the $T$ action.

\section{Local geometry around codimension two walls}

 Let $X$ be a smooth projective variety with a torus $T= (\CC^*)^n$ action and let $G$ be the associated graph as defined in the previous section. Let $L$ be a loop of $G$. We say that $L$ is local or simple if there exists a codimension two face $F$ of a polytope in $\mu(X)$ such that all the polytopes corresponding to the vertices $k$ of the loop contain $F$ as a face. Let $F_k$ denote the face corresponding to the vertex $k$ of the loop.

  For each $j \in J$ which is a vertex of $L$, let $\CC^*_j \subset T$  denote the isotropy group for points  in  $\begin{array}{ll} Y_j :=\{ x\in X; &  \mu (T\cdot x) \subseteq F_j \}\end{array} \subset X$.
  We will denote the components of $Y_j$ by $Y_{j,s}$.
  Notice that if $j$ and $j'$ are such that $Y_j \cap Y_{j'} \neq \emptyset$, then there exists a subtorus $T_F \simeq (\CC^*)^2$ such that $\CC^*_{j''}$ is a subgroup of $T_F$, for all $j''$ such that $F_{j''}$ is one of the faces containing $F$. Assuming such $j$ and $j'$ exist, we will denote by $Z_s$, $s \in B$ all components of the fixed locus for the $T_F$--action on $X$ which are subsets of $Y_j \cap Y_{j'}$ for some distinct $j, j'$ as above, and such that  $F\subseteq \mu(Z_s)$. We further denote by $Y_{j,s}$ the component of $Y_j$ which contains $Z_s$. Thus $Y_{j,s} \cap Y_{j',s}=Z_s$. We denote by $BY_{j,s} ^+$ and $BY_{j,s} ^-$ the subvarieties covered by all $\CC^*_j$--orbits starting and ending in $Y_{j,s}$, respectively.

\begin{proposition} \label{void intersection}
$$\begin{array}{ccc} BY_{j,s}^+ \cap BY_{j,s'}^+= \emptyset, & BY_{j,s}^+ \cap BY_{j,s'}^-= \emptyset  \mbox{ and} & BY_{j,s}^- \cap BY_{j,s'}^-= \emptyset \end{array}$$
whenever $Y_{j,s}\neq Y_{j,s'}$.

\end{proposition}

\begin{proof}

First notice that because $Y_j$ is smooth, $Y_{j,s} \cap Y_{j,s'} = \emptyset$. This implies that $BY_{j,s}^+ \cap BY_{j,s'}^+= \emptyset$ and $BY_{j,s}^- \cap BY_{j,s'}^-= \emptyset$.
We still have to prove that $BY_{j,s}^+ \cap BY_{j,s'}^-= \emptyset$.

  If $BY_{j,s}^+ \cap BY_{j,s'}^- \neq \emptyset$ consider $y \in BY_{j,s}^+ \cap BY_{j,s'}^- $.  So $\ol{\CC^*_j y} \cap Y_{j,s} \neq \emptyset$ and $\ol{\CC^*_j y} \cap Y_{j,s'} \neq \emptyset$.   The moment map  $\mu_j:X \to \RR$ associated to the action of $\CC_j^*$ on $X$ factors through $\mu:X \to \RR^n$ and a projection $g: \RR^n\to \RR$.
  As $g(F_j)=\mu_{j}(Y_{j,s})=\mu_{j}(Y_{j,s'})$ is a point in $\RR$, we conclude that $\mu_j(\lim_{t\rightarrow 0}ty) =\mu_j(\lim_{t\rightarrow \infty}ty)$. This could only happen if $y$ is a fixed point for $\CC^*_j$, which contradicts $Y_{j,s} \cap Y_{j,s'} = \emptyset$. The contradiction came from the assumption of the existence of $y \in BY_{j,s}^+ \cap BY_{j,s'}^-$, so   $BY_{j,s}^+ \cap BY_{j,s'}^-= \emptyset$.

\end{proof}

\begin{remark}\label{affine bundle} It is known in general that $BY_{j,s}^{\pm}$ is a $T$--equivariant affine bundle over $Y_{j,s}$ (see \cite{Bial}). Locally an affine bundle with a section and a torus action can be identified with a vector bundle with a linear torus action. We will show now that for $j$ and $j'$ distinct, $BY_{j,s}^{\pm} \cap BY_{j',s}^{\pm}$ is also a $T$--equivariant affine bundle over $Z_s$. For this, first notice that $BY_{j,s}^{\pm} \cap Y_{j',s} $ is an affine bundle over $Z_s$ because $BY_{j,s}^{\pm}\to Y_{j,s}$ is a $T$--equivariant affine bundle and $BY_{j,s}^{\pm} \cap Y_{j',s} $ is the part of weight zero in the restriction $BY_{j,s}^{\pm}|_{Z_s}$, for the action of $\CC^*_{j'}$. Furthermore, the following  is a Cartesian diagram
 \bea \diagram
{ BY_{j,s}^{\pm} \cap BY_{j',s}^{\pm}}\dto  \rrto && {  BY_{j,s}^{\pm}\cap Y_{j',s}  } \dto \\
{BY_{j,s}^{\pm}\times_{Y_{j,s}}(Y_{j,s} \cap BY_{j',s}^{\pm})} \rrto  &&  {BY_{j,s}^{\pm}\times_{Y_{j,s}}Z_s,}
         \enddiagram \eea
         where the vertical arrows are embeddings, and the horizontal arrows are given by
         \bea \diagram  { y } \rrto &&  \lim_{t'\to q'} t'y \enddiagram \eea
         for $t'\in \CC^*_{j'}$  and $q'= 0$ or $\infty$. The second row is the pullback of the affine bundle $Y_{j,s} \cap BY_{j',s}^{\pm}\to Z_s$ and so also
  $BY_{j,s}^{\pm} \cap BY_{j',s}^{\pm}$ is an affine bundle over $Y_{j',s} \cap BY_{j,s}^{\pm}$.
\end{remark}

\begin{proposition}\label{simple loop}

For any local loop $L$ corresponding to a codimension two face $F$ of $\mu(X)$, all components of $\underleftarrow{\lim}_{L}[X//T]$ have dimension smaller or equal with $\dim X- \dim T$. For any path  $ P \subset L$ all components of  $\underleftarrow{\lim}_{L}[X//T]$ except one are smooth.
The components of $\underleftarrow{\lim}_{ L}[X//T]$ of dimension smaller then $\dim X-\dim T$ are smooth.

\end{proposition}

\begin{proof}

 For $Z_s $ defined as above, let $\cX$ denote the deformation of $X$ to $N_{Z|X}$, the normal bundle of $Z = \cup_s Z_s$ in $X$. As $Z$ is a $T$--equivariant set, there is a  naturally induced action of $T$ on $\cX$. We will denote by $\cB Y_{j,s}^+$ and $\cB Y_{j,s}^-$ the induced deformations of $BY_{j,s}^+$ and $BY_{j,s}^+$ to the normal bundles of $Z_s$  in $BY_{j,s}^+$  and  $BY_{j,s}^+$, respectively.

  Keeping the notations above, let $U= \cup_{i\in I_L} U_i$ be the union of stable point sets for the stability conditions of the chambers containing $F$. Then $U \setminus( \cup_{j,s} (BY_{j,s}^+ \cup BY_{j,s}^-) )$ is the set of points which are stable for all these stability conditions.
 The stack quotient  $[   U \setminus( \cup_{j,s} (BY_{j,s}^+ \cup BY_{j,s}^-) )     //T] $ is a smooth open set in $\underleftarrow{\lim}_{ L}[X//T]$, of dimension $\dim(X) -n$. Moreover,
 the larger quotient $\underleftarrow{\lim}_{ L}[U \setminus( \cup_{j\not= j'} (BY_{j,s}^{\pm} \cap BY_{j',s'}^{\pm}) )//T] $ is locally isomorphic to one of $[X//T]_j$-s for $j\in J_L$, and thus also a smooth Deligne-Mumford stack.

 It remains to look for the other strata in $\underleftarrow{\lim}_{ L}[X//T]$. We will show that, for the purposes of this Proposition, we can replace $X$ by $N_{Z|X}$. For this we will need the following observations:

   $\bullet$ Among the subsets $BY_{j,s}^{\pm} \cap U$ of $U$ for all $(j,s)$, the only pairs that have nonempty intersection are of the form $BY_{j,s}^{\pm} \cap U$ and $BY_{j',s}^{\pm} \cap U$. Indeed, if $y \in BY_{j,s}^{\pm}  \cap BY_{j',s'}^{\pm} \cap U$, we may assume $j \neq j'$ due to Proposition \ref{void intersection}.
 Then for $t \in \CC^*_j$ and $t' \in \CC_{j'}^*$, and $q, q':=0$ or $\infty$ (depending on the $\pm$ sign choices above), we have  \bea \lim_{t\rightarrow q, t'\rightarrow q'} tt'y \in Y_j \cap Y_{j'},\eea as $Y_j$ and $Y_{j'}$ are $T$--equivariant.
  As $y \in U$, $F \subseteq \ol{\mu(Ty)}$ so there exists $s''$ such that $\lim_{t\rightarrow q, t'\rightarrow q'} tt'y \in Z_{s''}$. Due to Proposition \ref{void intersection}, $Y_{j',s''} = Y_{j',s'}$ and $Y_{j,s} = Y_{j,s''}$.

  $\bullet$
   The deformation of $BY_{j,s}^{\pm}\cap BY_{j',s}^{\pm}$ to the normal bundle of $Z_s$ in $BY_{j,s}^{\pm} \cap BY_{j',s}^{\pm}$ is locally trivial because $BY_{j,s}^{\pm} \cap BY_{j',s}^{\pm}$ is an affine bundle over $Z_s$.

  $\bullet$  Locally around any point in $Z_s$, the formal neighborhood of  $BY_{j,s}^{\pm} \cap BY_{j',s}^{\pm}$ in $X$ has a trivial deformation to the formal neighborhood of
  \bea  N_{Z_s|BY_{j,s}^{\pm} \cap BY_{j',s}^{\pm} } \hookrightarrow N_{Z_s|X}.\eea
  Here $N_{Z_s|BY_{j,s}^{\pm} \cap BY_{j',s}^{\pm} } $ and  $N_{Z_s|X}$ denote the normal bundles of $Z_s$ in $BY_{j,s}^{\pm} \cap BY_{j',s}^{\pm}$ and $X$, respectively. This deformation of formal neighborhoods is compatible with the construction of finite $G$-cover as in the proof of Proposition \ref{moduli}.
This last observation is relevant to the construction of the inverse limit $\underleftarrow{\lim}_{ L}[X//T]$ and its structure over $Z$.

After replacing $X$ by $N_{Z|X}$, the image of the moment map will keep its original structure around $F$, and by abuse of notation we will write $U\subseteq N_{Z|X}$ for the union of stable point sets for the stability conditions of the chambers containing $F$. If $T= T_F \oplus T'_F$, then $U$ is a $T'_F$--bundle over some smooth Deligne-Mumford stack $[U//T'_F]$, and the action of $T_F$ descends to $[U//T'_F]$.

By looking at the fibres of the vector bundle $[U//T'_F] \to [Z//T'_F]$, we can reduce the problem to the case of a linear action of $(\CC^*)^2$ on a vector space $V$. In this case it is esay to check.Let $V = \oplus V_i$ the decomposition of $V$ in vector spaces such that for each $i$ there is $\CC^*_i \subset (\CC^*)^2$ and $V_i$ is the fixed loci of $\CC^*_i$. There is a cyclic order on the set of indexes induced by the moment map. Let $O_{1,i_2}, O_{i_2,i_3},...,O_{i_k,1}$ a set of closures of $(\CC^*)^2$ orbits such that $O_{i_j,i_{j+1}} \cap O_{i_{j-1}i_j}\{V_{i_j} \setminus 0\} \neq \emptyset$ and $i_j$-th are in cyclic order. It can be easily seen that there exists a $(\CC^*)^2$ orbit $O$ which we can deform using a $\CC^*$ action on $V$ to $ \bigcup O_{i_j,i_{j+1}}$. By picking points $v_{i_j,i_{j+1}} \in  O_{i_j,i_{j+1}}$ we can find $v$ whose projection on each $V_{i_j} \oplus V_{i_{j+1}}$ is $v_{i_j,i_{j+1}}$. We can pick $\CC^*$ such that $\CC^* \oplus (\CC^*)^2$ contains  $ \bigcup O_{i_j,i_{j+1}}$. The only reason $\underleftarrow{\lim}_{L}[V/(\CC^*)^2]$ is might not be irreducible is because of the stack structure of $[V/(\CC^*)^2]$, however in this case all the other components of  $\underleftarrow{\lim}_{L}[V/(\CC^*)^2]$ are isomorphic to substacks of the component parameterizing the generic orbits.  But this implies that $\underleftarrow{\lim}_{L} [N_{Z|X}/T]$ has the same properties. As the limit of $BY_{j,s}^{\pm}$ under the deformation to the normal cone is locally isomorphic to $BY_{j,s}^{\pm}$ and the deformation to the normal bundle is compatible with the moment map we deduce that $\underleftarrow{\lim}_{L} [X/T]$ also has the same properties.

We will first discuss the case when $\mu(V)$ spans an angle of vertex $\mu(0)$, bounded by two lines $l_1$ and $l_2$. We will denote  $V_1 = \mu^{-1} (l_1)$ and $V_2 = \mu^{-1}(l_2)$, and by  $\CC^*_1$ and $\CC^*_2$ the sub-tori of $T$ which act trivially on  $V_1$ and $V_2$, respectively. Let $P_1 := [V_1 // (T/ \CC^*_1)]$ and $P_2 := [V_2// (T/ \CC^*_2)]$.

\begin{lemma}\label{smooth deformation to one orbit}

The natural morphism $\underleftarrow{\lim}_{}[V//T] \to P_1 \times P_2$ is smooth.

\end{lemma}

\begin{proof}
The quotient $\underleftarrow{\lim}_{}[V//T]$ parameterizes chains of torus orbit closures, with starting and ending points in $V_1$ and $V_2$, respectively. The morphism above is obtained by considering the start and endings of these chains.

Consider the decomposition of $V =  \oplus_{i=0}^k V_i$, where $V_0$ is the invariant locus under the action of the full torus $T$, and for each $i$ there exist $\CC^*_i \subset T$ such that $V_i$ is pointwise invariant under the action of $\CC^*_i$. As the quotient $\underleftarrow{\lim}_{}[V//T]$ is a fibration over $V_0$, for the purpose of this proof we may assume that $V_0$ is a point. The fibers of $\underleftarrow{\lim}_{}[V//T] \to P_1 \times P_2$ are of the form $\underleftarrow{\lim}_{}[V'//T] $ where $V'\cong V'_1\times V_2\times ...\times V_k\times V'_2,$ and
  $V'_i$ is the normalization of a one-dimensional orbit closure $ \ol{Tu_i}$ for each $i\in \{1,2\}$. So, to prove the lemma it will be enough to show that $\underleftarrow{\lim}_{}[V'//T]$ is smooth. Since $V'$ has a vector space structure on which $T$ acts linearly, we have thus reduced the problem to the case when $P_1\times P_2$ is a point.

  For any other two starting and ending orbits $V''_1$ and $V''_2$ there is a natural isomorphism of $V''\cong V''_1\times V_2\times ...\times V_k\times V''_2,$ and $V'\cong V'_1\times V_2\times ...\times V_k\times V'_2.$ This proves that $\underleftarrow{\lim}_{}[V//T] \simeq P_1 \times P_2 \times [V'//T]$.

We will assume that the indices of $V_i$-s are chosen such that if $k \geq i > j\geq 3$, then $\mu(V_j)$ is situated between $\mu(V_1)$ and $\mu(V_i)$.

The strata $\ol{BV_i^+}$ and $\ol{BV_i^-}$ are all smooth and, for any $i, i'$, either  $\ol{BV_i^+}\subseteq \ol{BV_{i'}^+}$ or $\ol{BV_{i'}^+}\subseteq \ol{BV_{i}^+}$ (and similarly for $\ol{BV_i^-}$) and any two $\ol{BV_i^+}$, $\ol{BV_j^-}$ intersect transversally.  Indeed, for suitable choices of the generators of $\CC^*_i$,  $\ol{BV_i^+} = V_2 \oplus(\bigoplus_{j\geq i} V_j)$ and $\ol{BV_i^-} =  \bigoplus_{1\leq j\leq i} V_j$.

We now construct a space $\tilde{V}$ by considering successively (weighted) blow-ups of $V$ along the (strict transforms of) all $\ol{BV_i^+}$ and $\ol{BV_i^-}$. The structure of the weighted blow-ups is described in the proof of Proposition \ref{moduli}. There is an open set in $\tilde{V}$ that plays the role of an universal family over $\underleftarrow{\lim}_{}[V//T] $. In particular, there is an induced action of $T$ on $\tilde{V}$ and
 a natural morphism to $\underleftarrow{\lim}_{}[V//T] $. Moreover, $\tilde{V}$ is a smooth Deligne-Mumford stack. The open set in $\tilde{V}$ made of points with finite stabilizers for the action of $T$ is smooth over $\underleftarrow{\lim}_{}[V//T] $. This implies that $\underleftarrow{\lim}_{}[V//T] $ is also smooth.

 The structure of $\tilde{V}$ follows from the following lemma, which also suggests how to view $\tilde{V}$ as a colimit of a category. This category can be described by looking at all natural contractions of $\tilde{V}$. Because $\tilde{V}$ as the colimit of this category, it has a natural morphism $\tilde{V} \to \underleftarrow{\lim}_{}[V//T] $.

 We use the notation $Bl_Y X$ for the blow-up of $X$ along $Y$  and $\tilde{Y}$ for the strict transform of $Y$.

\begin{lemma}

a) Let $Y_1$ and $Y_2$ be two smooth subvarieties of a smooth variety $X$. If $Y_1$ intersects $Y_2$ transversally, or if $Y_2 \subset Y_1$, then $Bl_{\tilde{Y}_1}Bl_{Y_2}X = Bl_{Y_1}X \times_X Bl_{Y_2}X$.
Also, $\tilde{Y}_1$ is smooth in both cases. Furthermore, if  $Y_2 \subset Y_1$ and $Y_3$ is a smooth subvariety intersecting $Y_2$ (and $Y_1$) transversally, then the strict transforms in $\tilde{Y}_1$ and $\tilde{Y}_3$ in $Bl_{Y_2}X$ are smooth and intersect transversally.

b) Let $f : X' \to X$ be the projection to the quotient through the action of a finite group $G$. We will assume that $G$ acts on $X'$ by reflections with respect to a smooth divisor  $D'$, meaning that for any point $x \in X'$ with nontrivial stabilizer, we have $x\in D'$ and for any $x \in D'$, we have  $\Stab(x)=G$. If  $Y' \subset X'$ is smooth and $G$ equivariant such that $D' \cap Y'$ is smooth and if $Z \subset f(D')= D'$ is smooth then $Z$ intersects $Y'$ transversely if and only if $Z$ intersects $f(Y)$ transversely.

\end{lemma}

\begin{proof}

a) The fact that $Bl_{\tilde{Y}_1}Bl_{Y_2}X = Bl_{\tilde{Y}_2}Bl_{Y_1}X$ is an observation that Thaddeus attributes to Thurston in \cite{thad}. This observation implies the first part of a). However the simplest proof comes from direct local computation. The rest of part a) could also be found in \cite{mcph} where it is used for similar reasons. Again this is a simple computation that we leave to the reader.

b) The proof is straightforward and we leave it to the reader.

\end{proof}

The fact that the universal family $\tilde{V}$ is smooth is a consequence of repeatedly applying the above lemma as at each step the loci of blow-up are either embedded one in the other or they intersect transversely. Part b) is used to reduce the case of the weighted blow-ups to the case of blow-ups along smooth varieties.

While we described the universal family  $\tilde{V}$ , the fibers of $\underleftarrow{\lim}_{}[V//T] $ are easier to describe, being just the weighted blow-up of a weighted projective space along a sequence of weighted projective spaces, each included in the next one in the sequence.

\end{proof}


\end{proof}

\begin{proposition}\label{tree to loop}

For any local loop $L$ around a codimension two face $F$, let $\Upsilon_j$ denote the tree obtained by cutting the edge $e_j$ of $L$. Then for each component $[X//T]_{\Upsilon_j,k} \subset \underleftarrow{\lim}_{\Upsilon_j}[X//T]$ there exists a component $[X//T]_{L,k} \subset \underleftarrow{\lim}_L[X//T]$ and deformations $\cX_{\Upsilon,k}$ and $\cX_{L,k}$ of $[X//T]_{\Upsilon_j,k}$ and $[X//T]_{L,k} $ over $\CC$, which are trivial over $\CC^*$, with a surjective morphism $f_k:\underleftarrow{\lim}_{\Upsilon_j}[X//T]_0 \to \underleftarrow{\lim}_L[X//T]_0$ between their fibers over $0$. The morphism $f_k$ is an inverse to the left of the the deformation of the natural embedding $\underleftarrow{\lim}_{L}[X//T] \hookrightarrow \underleftarrow{\lim}_{\Upsilon_j}[X//T]$.

Moreover, for each two components $[X//T]_{\Upsilon_j,k}$ and $[X//T]_{\Upsilon_j,k'}$ let $[X//T]_{\Upsilon_j,k,k',0}$ and $[X//T]_{\Upsilon_j,k',k,0}$ the limit at $0$ of  $([X//T]_{\Upsilon_j,k} \cap [X//T]_{\Upsilon_j,k'}) \times \CC^*$ in $\cX_{\Upsilon,k}$ respectively $\cX_{\Upsilon,k}$. Then $[X//T]_{\Upsilon_j,k,k',0} \simeq [X//T]_{\Upsilon_j,k',k,0}$ and $f_{k|[X//T]_{\Upsilon_j,k,k',0}} \simeq f_{k'|[X//T]_{\Upsilon_j,k',k,0}}$.

\end{proposition}

\begin{proof}

The component of maximum dimension in $\underleftarrow{\lim}_{L}[X//T] $ is also a component in $\underleftarrow{\lim}_{\Upsilon_j}[X//T]$ so we will have describe the morphism $f_k$ only for components whose generic point parameterizes more then a pair of orbits of maximum dimension.

We will reduce the problem using the same steps as in the Proposition \ref{simple loop} to the case of the action of $(\CC^*)^2 = T_F$ on the fibers of the normal bundle of $Z_s \in U$. Indeed the flat families are induced by the deformation of $U$ to the normal bundle of $Z_s$ as in the proof of the Proposition \ref{simple loop}.

We will assume that the generic point of the component $[X//T]_ {\Upsilon_j ,k,0}$ parameterizes union of orbits $O_1 \cup...\cup O_n$such that the boundary of $O_r$ goes via $\mu$ into a pair of half-lines $l_{i_r}, l_{i_{r+1}}$. The set ${i_1,...i_r} \subset {1,2,...m}$ where $l_1, l_2,..l_m$ are all half lines which contains points of the form $ \mu(x)$, where $\Stab(x)= \CC^* \subset T_F$. We assume that $i_1= 1$, $i_q > i_p$ if $q>p$ and $l_q$ is inside the angle generated by $l_{q-1}$ and $l_{q+1}$. The difference between $[X//T]_ {\Upsilon_j ,k,0}$ and $[X//T]_ {L ,k,0}$ is twofold. First, in the case of $[X//T]_ {\Upsilon_j ,k}$ $l_{i_1}$ doesn't have to be the same with $l_{i_{n+1}}$ and second, even when some points in $[X//T]_ {\Upsilon_j ,k}$ parameterizes sequence of orbits for which $\mu(O_1) \cap \mu(O_n) = l_{i_1}$ the orbits $O_1$ and $O_n$ might not have any boundary points in common outside $Z_s$. The morphism $f_k$ will be written as a composition $f_k= g_k \circ s_k $ where the image of $s_k$ will consists of union of orbits such that $l_{i_1}=l_{i_{n+1}}$.

 Let's us consider the half-plane $h_{i_n}$ bounded by the line containing $l_{i_n}$ and containing inside $l_{i_1}$. Let $N_{i_n}$ denote the sub-bundle of $N_{Z_s|U}$ which is the preimage of $h_{i_n}$ and $N_{i_n,i_1}$ the sub-bundle of $N_{Z_s|U}$ which is the preimage of the angle bounded by $l_{i_n}$ and $l_{i_1}$, and let  $M_{i_1}=\mu^{-1}(l_{i_1})$ be the sub-bundle of $N_{i_n,i_1}$. There is a linear projection $N_{i_n} \to N_{i_n,i_1}$ The preimage under this projection of $M_{i_1}$ is the sub-bundle $N_{i_1} \hookrightarrow N_{i_n}$ such that $\mu(N_{i_1})$ is maximal with the property that it doesn't intersect the interior of the angle  between $l_{i_n}$ and $l_{i_1}$. We want to check that this projection induces a morphism $\underleftarrow{\lim} [N_{i_n}/T_F] \to \underleftarrow{\lim} [N_{i_n,i_1}/T_F]$.  If $t$ is the largest number such that $l_t \subset h_{i_n}$, let $M_1 = \mu^{-1} (l_1)$,...,$M_t= \mu^{-1}(l_t)$ and $N_1,...,N_t$ denote the sub-bundles of $N_{i_n}$ such that they are maximal with the property that $\mu(N_1),...,\mu(N_t)$ don't intersect the interior of the angles bounded by $l_{i_n}$ and $l_1$,...,$l_t$. Let $N_{i_n,1},...,N_{i_n,t}$ denote the preimage of the angles bounded by $l_{i_n}$ and $l_1$,...,$l_t$ in $N_{i_n}$. Notice that there are also natural projections $N_{i_n} \to N_{i_n,t} \to...\to N_{i_n,t}$ and that the preimages of $M_q \subset N_{i_n,q}$ in $N_{i_n}$ is $N_q$ for any $q \in{1,...,t}$.

 We will prove inductively that there that the projections from $N_{i_n} \to N_{i_n,t} \to...\to N_{i_n,1}$ induce morphisms $\underleftarrow{\lim} [N_{i_n}/T_F] \to \underleftarrow{\lim}[N_{i_n,t}/T_F] \to...\to \underleftarrow{\lim} [N_{i_n,1}/T_F]$. Form the proof of Proposition \ref{simple loop} the construction of the universal family over $\underleftarrow{\lim} [N_{i_n}/T_F]$ starts with a weighted blow-up along $N_t$. As in the proof of Proposition \ref{moduli} we can realize this weighted blow-up by locally writing  $N_{i_n}$ as the quotient of a vector bundle $N_{i_n,G}$ by a finite group, and similarly for $N_{i_n,t}$ and $N_t$, and $M_t$. Notice that the projection $ N_{i_n} \to N_{i_n,t}$ lifts to a projection $ N_{i_n,G} \to N_{i_n,t,G}$ with the preimage of $M_{t,G}$  being $N_{t,G}$. Notice that there is a morphism between one of the GIT quotients $[Bl_{N_{t,G}} N_{i_n,G}/T_F]$ and one of the GIT quotients $[N_{i_n,t,G}/T_F]$. This immediately leads to a morphism $\underleftarrow{\lim} [N_{i_n}/T_F] \to \underleftarrow{\lim}[N_{i_n,t}/T_F]$ by following the proof of Proposition \ref{simple loop}. The composition of all  $\underleftarrow{\lim} [N_{i_n}/T_F] \to \underleftarrow{\lim}[N_{i_n,t}/T_F] \to \underleftarrow{\lim} [N_{i_n}/T_F] \to \underleftarrow{\lim}[N_{i_n,1}/T_F]$ is $s_k$.

 We start now the construction of $g_k$. Let $A_k$ denote the source of $g_k$, the moduli parameterizing union of orbits $O_1 \cup...\cup O_n$ which beside the usual relation between $O_1$ and $O_2$,...,$O_{n-1}$ and $O_n$ requires that $l_1$ is part of the boundary of $\mu(O_n)$. From the proof of Lemma \ref{smooth deformation to one orbit} we can conclude that $A_k = M_1 \times M_1 \times M$ for some smooth Deligne-Mumford stack $M$. $g_k$ could be any of the projections of $A_k \to  M_1 \times M$.

 The compatibility between the families of deformations for all components of $\underleftarrow{\lim}_{\Upsilon_j}[X//T]$ comes from the fact that they are all induced by the same deformation of $U$ into $N_{Z_s|U}$. The proof of the compatibility between the projections is also straightforward and we leave it to the reader.

 \end{proof}

 \section{Construction of the virtual class}

For a variety $X$ whose image under the moment map is given by a loop of chambers arranged around a codimension $2$ face, the only candidate for the virtual fundamental class of dimension $\dim X-\dim T$ of $A_{\dim X-\dim T}(\underleftarrow{\lim}_{}[X//T])$ is the fundamental class of the component of dimension $\dim X -\dim T$. We would now like  to see what happens if the graph $G$ contains other chambers beside the chambers around the codimension two face.

Let $G$ be the graph whose vertices correspond to the chambers and whose edges correspond to the walls of the chambers of $\mu(X)$. Let $\Upsilon$ be a maximal tree which is a subgraph of $G$. Notice that $ \underleftarrow{\lim}_{G}[X//T] \subset \underleftarrow{\lim}_{ \Upsilon} [X//T]$.
 The definition of a virtual class of  $\underleftarrow{\lim}_{ \Upsilon} [X//T]$ given in Remark \ref{colimit} can be extended to the colimit after the entire graph $G$ as follows.

With the notations from Remark \ref{colimit}, let $C_{\Upsilon} :=$ the normal cone of $\underleftarrow{\lim}_{\Upsilon}[X//T]$ in $ \prod_{j \in J_{\Upsilon}} [X//T]_j$, let $N_{\Upsilon}:=$ the pull-back on $\underleftarrow{\lim}_{\Upsilon}[X//T]$ of  the normal bundle of $
{\prod_{i \in I_{\Upsilon}}[X//T]_i} \hookrightarrow {\prod_{e_{ji}\in E(\Upsilon)}[X//T]_i}$, and $p: \underleftarrow{\lim}_{G}[X//T] \to  \underleftarrow{\lim}_{ \Upsilon}[X//T]$.
\begin{definition}\label{virtual class}
 We define the virtual class  $v(\underleftarrow{\lim}_{G}[X//T])$ as the union of components of dimension $\dim X-n+\rank(N_{\Upsilon})$ of $ [p^*C_{\Upsilon}]$  as a class in the Chow group $A_{\dim X-n+\rank(N)}(p^*N_{\Upsilon})$. Via the identification  $A(p^*N_{\Upsilon}) = A(\underleftarrow{\lim}_{G}[X//T])$, we consider $v(\underleftarrow{\lim}_{ }[X//T])$ as a class in $A_{\dim X-n}(\underleftarrow{\lim}_{G}[X//T])$.
\end{definition}

We will say that 2 cones,  $C_0  $ and $C_1$ over a stack $Z$ are equivalent if there exists  a flat  family  of cones respectively  stack cones over $Z$ $f:C \to \CC$ such that $f^{-1}(0)=C_0\oplus E_0$ and $f^{-1}(1)=C_1 \oplus E_1$ for two vector bundles $E_0$ and $E_1$.

\begin{theorem}

The definition above is independent of the choice of $\Upsilon$.

\end{theorem}

\begin{proof}

We will first look to the case when $T= (\CC^*)^2$.

There is a geometric realization of the graph $G$ obtained by identifying the nodes of $G$ with the weight center of the corresponding polytopes and the edges with the union of pairs of segments joining the weight center of the polytopes with the weight center of their faces.  Because we work with $(\CC^*)^2$, all the polytopes are two dimensional so any loop $L$ in $G$ divides $\RR^2$ in two disjoint sets, one of which is bounded and contractible -- the interior of the loop.  We denote by $G_L$ the subgraph of $G$ containing only vertexes and edges inside $L$. Applying Definition \ref{virtual class} to $G_L$ (i.e. using only maximal trees in $G_L$), we obtain a class $\underleftarrow{\lim}_{ G_L} [X/T]$. We will first prove that this class is independent of the choice of a maximal tree in $G_L$.

Consider two choices of maximal trees $\Upsilon$ and $\Upsilon'$ in $G_L$, the corresponding cones $C_{\Upsilon}$ and $C_{\Upsilon'}$ and bundles $N_{\Upsilon}$ and $N_{\Upsilon'}$ from Definition \ref{virtual class}, the morphisms $p: \underleftarrow{\lim}_{G}[X//T] \to  \underleftarrow{\lim}_{\Upsilon} [X//T]$ and $p': \underleftarrow{\lim}_{G}[X//T] \to  \underleftarrow{\lim}_{ \Upsilon'}[X//T]$.  We claim that there is a flat deformation
\bea \diagram   \cC \dto \rto & {F\times \CC } \dlto \\  \CC \enddiagram \eea
of the subcone $ p^*C_{\Upsilon} \oplus p'^*N_{\Upsilon'} \mbox{ into }  p'^*C_{\Upsilon'}\oplus p^*N_{\Upsilon}$ inside $F:= p^*N_{\Upsilon}\oplus p'^*N_{\Upsilon'}$, whose restriction over any substack of the zero section $Z_0:=\underleftarrow{\lim}_{ G_L} [X//T]$ of $F$ is also flat over $\CC$. We say that $\cC$ is flat over $\CC$ relative to $Z_0$.

To prove this claim, we will proceed inductively, showing that if the statement is true for all loops $L$ contained in the interior of a loop $L_1$, then the statement will also be true for $L_1$.
In the case when the loop $L$ is minimal (i.e. it doesn't contain any vertex of $G$ in its interior) it follows from Proposition \ref{simple loop} that the virtual class of $G_L$ is independent of the choice of $\Upsilon$.


We first need to show that if $G_{L_1} $ is obtained from $G_{L}$ by gluing some trees at certain vertices, then the claim above is true for $G_{L_1}$ if it is true for $G_{L}$. Let $\Upsilon$ and $\Upsilon_1$ be two maximal trees for  $G_{L}$ and $G_{L_1} $, respectively, and assume that $\Upsilon_1$ is obtained by gluing the tree $\Upsilon''$ to $\Upsilon$. The relation between $v(\underleftarrow{\lim}_{G_L}[X//T])$ and $v(\underleftarrow{\lim}_{G_{L_1}}[X//T])$ is essentially a consequence of the functoriality of the Gysin pull-back for the lowest row of the following sequence of Cartesian diagrams
\bea\diagram
 \underleftarrow{\lim}_{ \Upsilon_1} [X//T] \rto  \ddto   & \underleftarrow{\lim}_{ \Upsilon} [X//T]   \times      {\prod_{j\in J_{\Upsilon''}}[X//T]_j}    \rto \dto & \prod_{j \in J_{\Upsilon_1}}[X//T]_j \dto\\
&{\prod_{i\in I_{\Upsilon}}[X//T]_i} \times {\prod_{j\in J_{\Upsilon''}}[X//T]_j} \rto \dto & {\prod_{e_{ji}\in E(\Upsilon)}[X//T]_i} \times {\prod_{j\in J_{\Upsilon''}}[X//T]_j}\dto \\
{\prod_{i \in I_{\Upsilon_1}}[X//T]_i} \rto & {\prod_{i \in I_{\Upsilon}}[X//T]_i} \times {\prod_{e_{ji} \in E(\Upsilon'')}[X//T]_i} \rto & {\prod_{e_{ji}\in E(\Upsilon_1)}[X//T]_i}
\enddiagram
\eea
However, due to the restriction to the component $\underleftarrow{\lim}_{G_L} [X//T] $ of $\underleftarrow{\lim}_{ \Upsilon} [X//T]$ in Definition \ref{virtual class}, the arguments related to the functoriality of the Gysin map will need to be reworked  in this case at the level of cycles/ cones, normal bundles.

For this we will need the following lemma:

\begin{lemma}\label{cone def}

 a) Assume $Z \subset Y \subset X$ are smooth varieties such that $N_{Z|X} = N_{Z|Y} \bigoplus (N_{Y|X})_{|Z}$. For a morphism $\phi :S \to X$, let $\phi_Y :S_Y= S \times_X Y \to Y$, and $\phi_Z: S_Z = S \times_X Z \to Z$.

 $C_{S_Z|S}  \subset \phi^*_Z(N_{Z|X})$ and $\phi^*_Z(N_{Z|X})$ is a vector bundle over $\phi^*_Z(N_{Z|Y})$. As such, $\phi^*_Z(N_{Z|X})$ admits a natural $\CC^*$--action
 whose fixed locus is $\phi^*_Z(N_{Z|Y})$, and
\bea\lim_{t\rightarrow 0} t C_{S_Z|S}= C_{S_Z|C_{S_Y|S}}.\eea
b) Let  $(p,\pi): \cC \to Z_0 \times \CC$ be a morphism of schemes and $s: Z_0 \times \CC \to \cC$ a section of $(p, \pi)$. For any subscheme $S \subset Z_0$, let   $\cC_S$ denote the family over $\CC$ whose fiber at $t \in \CC$  is $C_{s^{-1}(S \times t)|\pi^{-1}(t)}$.   If $\cC$ is flat relative to $Z_0$ (in the sense described above), then the family  $\cC_S$  is flat over $\CC$.

\end{lemma}

\begin{proof}

Let $q_1: M\to \CC$ be the deformation of $S $ to $C_{S_Y|S}$. Thus $M$ is an open set in $Bl_{S_Y \times \{0\} }S \times \CC$ and there is a natural embedding of $S_Z \times \CC$ into $M$. Let $p_2: N\to \CC$ be the deformation of $M$ to the normal cone $C_{S_Z \times \CC|M}$ and let $p_1: N\to \CC$ be the projection induced from $q_1$. We claim that $p_1^{-1}(0)$ is the deformation of $C_{S_Y|S}$ to the cone of $S_Z$ in $C_{S_Y|S}$.  Indeed, $M$ is an open set in $\Proj(\bigoplus_{n\geq0}(I_{S_Y}+(t))^n$. We can lift our discussion to $\Spec(\bigoplus_{n\geq 0}(I_{S_Y}+(t))^n$. Then $N$ could be replaced by an open set in $Bl_{\Spec{\oplus_nJ^n}}\Spec((\bigoplus_n(I_{S_Y}+(t))^n)[u])$, where $J$ is the ideal generated by $u$, the ideal of $S_Z$, and $\bigoplus_{n>0}(I_{S_Y}+(t))^n$. The claim above is equivalent to $J^n \cap (t)= (t) J^n$ which is straightforward. From here we also conclude that $p_2^{-1}(0)=C_{S_Z\times \CC|M}$ is a deformation of $C_{S_z|S}$ to the cone of $S_Z$ in $C_{S_Y|S}$.

Now let $M'$ be the deformation of $X$ to the normal bundle of $Y$ in $X$ and $N'$ the deformation of $M'$ to the normal bundle of $Z\times \CC$ in $M'$. There is a natural embedding of $C_{S_Z\times \CC|M} \subset N_{Z\times \CC|M'} \times_{Z\times \CC} S_Z\times \CC$. There is a natural isomorphism $f:N_{ Z\times\CC |M'} \to N_{Z|X} \times \CC$ (here we use $N_{Z|X} = N_{Z|Y} \bigoplus (N_{Y|X})_{|Z}$). This isomorphism however is not an extension of the canonical isomorphism $N_{ Z\times (\CC -\{0\})|M'_{|(\CC -\{0\})}} \cong N_{Z|X} \times (\CC -\{0\})$,  but the restriction of $f$ to $N_{ Z\times (\CC -\{0\})|M'_{|(\CC -\{0\})}}$ differs from the canonical isomorphism by multiplication with $\CC^*$ on $N_{Z|X} \times (\CC -\{0\})$.

Point b) follows by straightforward calculations.

\end{proof}

Returning now to the induction step, let $\Upsilon$ and $\Upsilon'$ be two maximal trees in $G_{L}$ which lead to the same definition of the virtual class. In a first instance we will consider two trees $\Upsilon_1$ and $\Upsilon_1'$ obtained by gluing a tree $\Upsilon''$ to the same vertex in  $\Upsilon$ and $\Upsilon'$.
The Lemma above can be applied to show that the virtual classes associated to $\Upsilon_1$ and $\Upsilon_1'$ are the same. We will denote by $G'_{L}$ the graph obtained from $G_{L}$ by gluing the tree $\Upsilon''$ to it.

Indeed we apply part a) of the Lemma twice, first when $Z \hookrightarrow Y \hookrightarrow X$ is \bean \label{exact sequence1} {\prod_{i \in I_{\Upsilon_1}}[X//T]_i} \hookrightarrow {\prod_{i \in I_{\Upsilon}}[X//T]_i} \times {\prod_{e_{ji} \in E(\Upsilon'')}[X//T]_i} \hookrightarrow {\prod_{e_{ji}\in E(\Upsilon_1)}[X//T]_i}\eean  and $S_Z \hookrightarrow S_Y \hookrightarrow S$ is
\bean \underleftarrow{\lim}_{ \Upsilon_1} [X/T] \hookrightarrow        \underleftarrow{\lim}_{ \Upsilon} [X/T]   \times      {\prod_{j\in J_{\Upsilon''}}[X//T]_j}                               \hookrightarrow  \prod_{j \in J_{\Upsilon_1}}[X//T]_j,\eean
then also for the case when $\Upsilon$ is replaced by $\Upsilon'$.
We obtain a deformation of $C_{\Upsilon_1}$ to $C_{\underleftarrow{\lim}_{ \Upsilon_1} [X/T]|(C_{\Upsilon} \times {\prod_{j\in J_{\Upsilon''}}[X//T]_j} )} $ inside $N_{\Upsilon_1}$. After pull-back to $\underleftarrow{\lim}_{ G'_{L}} [X//T]$, this induces a deformation of  $p^*_1C_{\Upsilon_1}\oplus p'^*_1N_{\Upsilon'_1}$ to $C_{\underleftarrow{\lim}_{G'_L} [X/T]|((p^*C_{\Upsilon} \oplus p'^*_1N_{\Upsilon'_1})\times {\prod_{j\in J_{\Upsilon''}}[X//T]_j} )}$ inside $p^*_1N_{\Upsilon_1}\oplus p'^*_1N_{\Upsilon'_1}$. (Here by abuse of notation, we identify some objects on $\underleftarrow{\lim}_{ G_{L}} [X//T]\times  \prod_{j \in J_{\Upsilon''}}[X//T]_j$ with their restriction on $\underleftarrow{\lim}_{ G'_{L}} [X//T]$).

By induction, there is a deformation of $ (p^*C_{\Upsilon}\oplus p'^*N_{\Upsilon'}) \times \prod_{j \in J_{\Upsilon''}}[X//T]_j$ to $ (p'^*C_{\Upsilon'}\oplus p^*N_{\Upsilon}) \times \prod_{j \in J_{\Upsilon''}}[X//T]_j$ in $(p^*N_{\Upsilon}\oplus p'^*N_{\Upsilon'})\times \prod_{j \in J_{\Upsilon''}}[X//T]_j$. This extends to a deformation $\cC$  of $ (p^*C_{\Upsilon}\oplus p_1'^*N_{\Upsilon'_1}) \times \prod_{j \in J_{\Upsilon''}}[X//T]_j$ to $ (p'^*C_{\Upsilon'}\oplus p_1^*N_{\Upsilon_1}) \times \prod_{j \in J_{\Upsilon''}}[X//T]_j$. Indeed, from (\ref{exact sequence1}),
the bundle $N_{\Upsilon_1} $ splits as a direct sum of the pull-back of $N_{\Upsilon''}$ and $N_{\Upsilon}$.
 We now apply part b) of the Lemma where $Z_0 = \underleftarrow{\lim}_{ G_{L}} [X//T]\times  \prod_{j \in J_{\Upsilon''}}[X//T]_j$ and $S= \underleftarrow{\lim}_{ G'_{L}} [X//T]$ .
This finishes the induction step in this case.

Let $\Upsilon$ be a maximal tree in $G_{L_1}$.  By adding an edge $e$ (whose vertices are in $\Upsilon$) to $\Upsilon$  we form a loop $L$ which is interior to ${L_1}$. We denote by $\Upsilon_{L}$ the subgraph of $\Upsilon$ which is interior to $L$. We first notice that this graph is a maximal tree in $G_{L}$. To check that $\Upsilon_{L}$ is a tree (connected), it is enough to consider two vertices in $\Upsilon_{L}$ and consider the path that joints them in $\Upsilon$. Any connected part of the path that is outside $L$ will intersect $L$ at two points and as such can be replaced by a path in $L$ so in the end we construct a path between the two vertices contained in $\Upsilon_{L}$. Notice that $\Upsilon$ contains all the vertices in $G_{L_1}$, so $\Upsilon_{L}$ contains all vertices in $G_{L}$ so it is maximal. Consider now $\Upsilon'$ a new maximal tree obtained by replacing an edge of $L$ with the edge $e$. By induction, the virtual class associated to $\underleftarrow{\lim}_{ G_{L}} [X//T]$ is the same whether its construction involves $\Upsilon_{L}$ or $\Upsilon'_{L}$. But then the virtual class of $$\underleftarrow{\lim}_{ G_L} [X//T]$$ is the same whether it is defined using $\Upsilon $ or $\Upsilon'$, as they are obtained from $\Upsilon_{L}$ and $\Upsilon'_{L}$ by gluing the same trees to some of their vertices.

We will prove now the general case,  $T = \CC^n$.  As the proof for quotients by $(\CC^*)^n$ for $n >2$ differs from the proof for $n=2$, we will need a few additional notions about complexes of polytopes:

\begin{definition}(see \cite{ziegler})

Let $\cC$ be a pure n-dimensional polytopal complex. A shelling of $\cC$ is a linear ordering $F_1,F_2,...,F_s$ of the facets of $\cC$ such that either $\cC$ is 0-dimensional, or it satisfies the following conditions:

i) The boundary complex $ \cC(\partial F_1) $ of the first facet $F_1$ has a shelling.

ii)For $1 < j< s$ the intersection of the facet $F_j$ with the previous facets is nonempty and is a beginning of a shelling of the $(n-1)$-dimensional boundary complex of $F_j$,that is,
$$F_j \cap( \cup_{i=1} ^{j-1} F_i)= G_1 \cup G_2 \cup...\cup G_r$$ for some shelling $G_1, G_2,....,G_r,...,G_t$ of $\cC(\partial F_j)$ and $1 \leq r \leq t$. (In particular, this requires that $F_j \cap( \cup_{i=1} ^{j-1} F_i)$ has a shelling, so it has pure $(n-1)$-dimensional, and connected for $n>1$.)

A polytopal complex is shellable if it is pure and has a shelling.

\end{definition}

\begin{remark}
With the notations above if $l,t \in \{1,...,r\}$ and $G_l$ and $G_t$ have a codimension one face in common then all the faces containing $G_l\cap G_t$ are between $F_1,...,F_s$. This is an immediate consequence of the fact that $F_j \cap( \cup_{i=1} ^{j-1} F_i)$ is a pure (n-1)-dimensional complex.

\end{remark}

\begin{theorem}(Bruggesser and Mani)
Polytopes are shellable. For any facet $F$ of $P$ there is a shelling of $P$ that ends with $F$.
\end{theorem}

  \begin{lemma}

  For any polytopal subdivision of a polytope has a refinement which is shellable

  \end{lemma}

  \begin{proof}

  For any subdivision of a polytope $P \subset \RR^n $ we can construct inductively a polytope $P' \subset \RR^{n+1}$ such that $P$ is a face of $P'$ and the polytope complex given by all the facets of $P'$ but $P$ is combinatorially equivalent to a refinement of the subdivision of $P$. We will proceed in steps such that at the  $k$-th step we construct a a polytope $P'_k$ such that the $k$- skeleton of the polytopal complex  $\cC(P'_k) $  is equivalent to a refinement of the $k$-skeleton of the subdivision of $P$. The key fact needed for this construction is the following simple observation:

  Let $P'_k$ be a polytope in $\RR^{n+1}$ and $ x \in \partial P'_k$. There is a unique complex of polytopes having as facets polytopes with one vertex $x$ and all the other vertexes being vertexes of $P_k'$. The $n$-skeleton of this complex will be denoted by $\cC(P'_k,x)$. We can move x to a new point $x'$ outside $P'_k$ such that if $P''_k$ is the polytope having as vertexes the vertexes of $P_k'$ and $x'$, $P_k''$ is combinatorially equivalent to $\cC(P'_k,x)$.

  We start by constructing  $P'_0$ such that $P$ is a facet of $P_0'$ and the vertexes of $P_0'$ project on the 0-dimensional faces of the subdivision of $P$ via the projection $p$. We consider the complex $\cC_0$ consisting in the subdivision of $\cC(\partial P_0')$ given by the preimage under projection map of all the faces of the subdivision of $P$. The preimage of a $1$-dimensional face of the subdivision of $P$ via $p$ is a union of 1 dimensional cells in the subdivision of $P$. We further refine the complex $\cC_{01}$ by adding extra points on each 1-dimensional face of $\cC_0$ which is not an edge of $P_0'$ and is of the form $p^{-1}(F)$ for $F$ a 1-dimensional face of $P$. Using the previous observation we can move all  $0$-dimensional faces of $\cC_{0,1}$ such that we obtain a new polytope $P_1'$ with the 1-dimensional skeleton of $\cC(\partial P_1')$ combinatorially equivalent to the one dimensional skeleton of $\cC_{01}$. We complete $\cC(\partial P_1')$ to a polytopal complex $\cC_1$ which is combinatorially equivalent to $\cC_{01}$. We introduce new points on all the $2$-dimensional faces of $\cC_1$ which are not faces of $\cC( \partial P_1')$ and construct a new complex $\cC_{12}$ which can be deformed to a polytope $P_2'$ and we continue the procedure till we construct $P_n'=P'$.

   We apply now the previous theorem to $P'$ and consider a shelling that ends with $P$. Via projection this corresponds to a shelling of a subdivision of $P$.


\end{proof}

We now return to the quotients by $T=(\CC^*)^n$ for $n>2$. For each polytope $P \subset \mu(X)$ which can be written as a union of chambers we consider a further decomposition of each chamber in polytopes such that the complex obtained is shellable. We construct a dual cell decomposition $D$ of the polytope $P$ whose zero cells are the weight centers of the polytopes in the decomposition of $P$  and each of its $k$-dimensional faces $F^{\vee}$ corresponds to a choice of an $n-k$ dimensional face $F$ in the polytope complex and the vertexes of $F^{\vee}$ are the weight center of its adjacent polytopes. As each 0-dimensional cell is inside a chamber it corresponds to a quotient of $X$ by $T$. Each $1$-dimensional cell corresponds either to a flip in the case we cross a wall of a chamber or to an isomorphism if the 1-dimensional cell is contained in a chamber. In this way we replaced our graph $G$ by a new graph $H$ with a natural surjective map $H \to G$, but as all new arrows are isomorphisms we didn't change in any significant way our problem.

 Let $\Upsilon$ and $\Upsilon'$ be two maximal trees in $H$ which differ by one edge. We will have to prove that the virtual classes defined using $\Upsilon$ or $\Upsilon'$ are the same. Let $F_1,...,F_n$ be an ordering of the polytopes in the subdivision of $P$ coming from the shellability of the complex. We will proceed by induction and at the $k$-th step of induction we will assume that all maximal trees in $F_1 \cup F_2 \cup... \cup F_{k-1}$ lead to the same virtual class. For $k=1$ or $k=2$ there is nothing to prove as there is no loop as two polytopes contains a unique 1-dimensional  cell. Let $F_k \cap( \cup_{i=1} ^{k-1} F_i)= G_1 \cup G_2 \cup...\cup G_r$. We will denote by $c_1,...,c_r$ the dual $1$-dimensional cells. Again using induction we will assume that if $L \subset \Upsilon \cup \Upsilon'$ doesn't contains any of $c_i, c_{i+1},...,c_r$ $\Upsilon$ and $\Upsilon'$ lead to the same virtual class and we will check that the same is true if $L$ doesn't contains any of $ c_{i+1},...,c_r$.
There exists $G_t$  for $t<i$ such that $G_t \cap G_i= G_{ti} \neq \emptyset$. Let $L_{ti}$ be the simple loop around $G_{ti}$. We will show that we can make transformations of $\Upsilon$ into another maximal tree that defines the same virtual class and is missing only one edge of $L_{ti}$. If $\Upsilon$ doesn't contain both $c_i$ and $c_t$ we just arbitrary replace edges of $\Upsilon \cap$ with edges of $L_{ti}$ such that we do not add both $c_i$ and $c_t$. By induction this procedure doesn't lead to a change in the definition of the virtual class. The case $\Upsilon$ contains  both $c_i$ and $c_t$ needs one extra construction. We divide in each of the polytopes $ F_j \subset \cup_{i=1}^{k-1}F_i$ such that $G_{it}\subset \partial F_j$ in two polytopes $F_j=F_j^1 \cup F_j^2$, $G_{it} \subset F_j^1$, and $G_{it} \nsubseteq F_j^2$. We can now replace $\Upsilon$ by a new tree $\Upsilon^1$ obtained from $\Upsilon$ by adding the path going through all $F_j^1$ for all $j \neq t$. As $\underleftarrow{\lim}_H [X//T] \hookrightarrow \underleftarrow{\lim}_{\Upsilon }[X//T] $ is the composition of  $\underleftarrow{\lim}_H [X//T] \hookrightarrow \underleftarrow{\lim}_{\Upsilon^1 }[X//T] $ and the projection  $\underleftarrow{\lim}_{\Upsilon_1} [X//T] \rightarrow \underleftarrow{\lim}_{\Upsilon }[X//T] $ the virtual class obtained using the trees $\Upsilon$ or $\Upsilon^1$ are the same. Now the loop around $G_{jt}$ is again missing only one edge. As we have seen in the case $T=(\CC^*)^2$ if two trees differ by an edge and their union contains a small loop they lead to the same virtual class.

\end{proof}

\section{Tautological rings and their properties}

\begin{definition}

Let $P_t = \cup_{i=1}^t F_i$ where $F_i$ are part of a shelling as above. As before $G_{P_t}$ denotes the graph determined by the $0$-dimensional skeleton of the center of weights of $F_i$ with $i<t$ and the $1$-dimensional skeleton that joints them.
We define the tautological ring $T(\underleftarrow{\lim}_{ G_{P_t}}[X//T])$ of $\underleftarrow{\lim}_{ G_{P_t}}[X//T]$  as the image under the homomorphism
\bea  A^*(\underleftarrow{\lim}_{ G_{P_t}}[X//T]) & \to  & A_*(\underleftarrow{\lim}_{ G_{P_t}}[X//T]) \\
\alpha & \to &  \alpha \cap v(\underleftarrow{\lim}_{ G_{P_t}}[X//T]) \eea
of the subring of $A^*(\underleftarrow{\lim}_{ G_{P_t}}[X//T])$ generated by the pullbacks of all the classes in $A^*([X//T]_{F_s})$ for $s<t$ .

\end{definition}

\begin{proposition} \label{push and pull}

Whenever $j<t $ there are natural maps $p^*:T(\underleftarrow{\lim}_{ G_{P_j}}[X//T]) \to T(\underleftarrow{\lim}_{ G_{P_t}}[X//T])$ and
$p_*:T(\underleftarrow{\lim}_{ G_{P_t}[X//T]}) \to T(\underleftarrow{\lim}_{ G_{P_j}}[X//T])$, with $p^*$ a ring morphism.

\end{proposition}

\begin{proof}

In the case when $G_{P_t}$ is obtained from $G_{P_j}$ by gluing some trees at the vertices of $G_{P_j}$, both $p^*$ and $p_*$ are well defined because of  the commutativity of the composition of Gysin maps respectively projection formula. If $G_{P_j}$ is obtained from $G_{P_t}$ by  gluing a simple loop $L''$, by cutting one edge of the loop which is not in $G_{P_j}$ and using Proposition \ref{tree to loop} we can can reduce the problem to the case of gluing a tree. Indeed let $\Upsilon$ denote the tree obtained from $L''$ by cutting an edge and $\Upsilon'$ the tree given by $\Upsilon \cap G_{P_j}$. Using Proposition \ref{tree to loop} we obtain two families of deformations of $[X//T]_{G_{P_j},\Upsilon,k} = \underleftarrow{\lim}_{ G_{P_j}}[X//T] \times_{[X//T]_{\Upsilon',k} } [X//T]_{\Upsilon,k}$ and  $[X//T]_{G_{P_j},L'',k} = \underleftarrow{\lim}_{ G_{P_j}}[X//T] \times_{[X//T]_{\Upsilon',k} } [X//T]_{\Upsilon,k}$ with a natural morphism
\bea f_k:[X//T]_{G_{P_j},\Upsilon,k,0} \to [X//T]_{G_{P_j},L'',k,0} \eea  between their fibers over $0$. First we will verify that there exist isomorphisms \bea A([X//T]_{G_{P_j},\Upsilon,k}) \simeq A([X//T]_{G_{P_j},\Upsilon,k,0}) \mbox{  and  } A([X//T]_{G_{P_j},L'',k}) \simeq A([X//T]_{G_{P_j},L'',k,0}) \eea induced by the family of deformations for all components $[X//T]_t$ whose generic point doesn't parameterize irreducible orbit of $T$ .

The check boils down to the following to lemma:

\begin{lemma} \label{affine bundle}

If A is an affine bundle over $Z$ with a $\CC^*$ action and a zero section $Z \subset A$, and $N_Z$ denotes the normal bundle of $Z$ in $A$, then there is a natural isomorphism of Chow groups $A([A \setminus Z/ \CC^*]) \simeq A([N_Z \setminus Z/ \CC^*])$ induced by the deformation to the normal cone.

\end{lemma}

\begin{proof}

Both $A([A \setminus Z/ \CC^*])$ and $A([N_Z \setminus Z/ \CC^*])$ are equal with $A(Z)[X]/ P(X)$ where $P(X)$ is the equivariant top Chern class of $N_Z$ (see \cite{noi1}).

\end{proof}

The next lemma part a) reduces easily to Proposition 6.7  from  \cite{fulton}.It can be used to effectively to compute the "virtual" intersection ring of the pull-back of a weighted blow-up.  Part b) is a  slight generalization of Theorem 1 from \cite{keel} and it shows the result of such a computation under some additional assumptions which makes the final formula look nicer. As we don't need the explicit formula part b) is not essential for our proof. The proof of part b) is the same proof as in \cite{keel} once we know Lemma \ref{affine bundle}.

\begin{lemma} \label{keel}

  Let $Z \subset Y$ be an imbedding of smooth varieties $K$ the kernel of the Gysin map $i^!:A(Y) \to A(Z)$ . Let $Bl^w_Z Y$ denote a weighted blow-up of $Y$ along $Z$.

a)A class $ \alpha \in v\cap A^*(Bl^w_Z Y \times_Y T)$ is $0$  if and only if its image in $A(T)$ is $0$ and $\alpha E=0$.

b)Assume $i^!$ is surjective. There exists a natural action of $\CC^*$ on of the normal bundle of $Z$ $N_Z$ and we will denote by  $P(E)$ a polynomial with coefficients in $A(Y)$ whose free term is $Y$ and whose pullback to $Z$ is the equivariant top Chern class of $N_Z$
  Then $A(Bl^w_Z Y)= A(Y)[E]/I$ where  $I$ is generated by $K \cdot E$ and $P(E)$. Moreover for any $\phi: T \to Y$ birational and $K_T$ denotes the kernel of Gysin map $A(T) \to A(Z \times_Y T)$ then if $ v= \phi^!([Bl^w_Z Y])$ then $A(T)[E]/(P_N(E), EK)$ is a subring of $v\cap A^*(Bl^w_Z Y \times_Y T)$, where we identified the classes in $A(Y)$ with their pull-back in $A(T)$.

  \end{lemma}

Indeed as all the components  $[X//T]_{G_{P_j},\Upsilon,k}$  which don't contain points parameterizing irreducible orbits can be described in terms of weighted blow-ups of sequences of weighted projective sub-bundles and fiber products of those. So their Chow groups can be computed by repeatedly applying Lemma \ref{affine bundle} and Lemma \ref{keel}.

In conclusion there exists a morphism $f: A( \underleftarrow{\lim}_{ G_{P_j}}[X//T] \times_{\underleftarrow{\lim}_{\Upsilon'}[X//T]} \underleftarrow{\lim}_{\Upsilon}[X//T]) \to A( \underleftarrow{\lim}_{ G_{P_j}}[X//T] \times_{\underleftarrow{\lim}_{\Upsilon'}[X//T]} \underleftarrow{\lim}_{L''}[X//T])$.

 As $ v(\underleftarrow{\lim}_{L''}[X//T])=[\underleftarrow{\lim}_{L''}[X//T]]$ we conclude that $$f(v(\underleftarrow{\lim}_{ G_{P_j}}[X//T] \times_{\underleftarrow{\lim}_{\Upsilon'}[X//T]} \underleftarrow{\lim}_{\Upsilon}[X//T])) = v(\underleftarrow{\lim}_{ G_{P_j}}[X//T] \times_{\underleftarrow{\lim}_{\Upsilon'}[X//T]} \underleftarrow{\lim}_{L''}[X//T]).$$

   If $p_{\Upsilon}: \underleftarrow{\lim}_{ G_{P_j} \cup \Upsilon} [X//T] \to \underleftarrow{\lim}_{ G_{P_j}}[X//T]$ and $p_{L''}: \underleftarrow{\lim}_{ G_{P_j}\cup L''} [X//T] \to \underleftarrow{\lim}_{ G_{P_j}}[X//T]$ denote the two natural projections it is easy now to check that $p_{\Upsilon *} = f \circ p_{L'' *}$. The proposition can now be proved using projection formula for our case when $G_{P_t}$ is obtained from $G_{P_j}$ by gluing a simple loop. In general the proposition can be solved inductively by gluing trees and simple loops.

\end{proof}

\begin{remark} \label{compute}
We can construct the tautological ring  $T(\underleftarrow{\lim}_{ G}[X//T])$ inductively using Lemma \cite{keel} every time we extend a graph by adding a new vertex and using the morphism $f: T( \underleftarrow{\lim}_{ G_{P_j}}[X//T] \times_{\underleftarrow{\lim}_{\Upsilon'}[X//T]} \underleftarrow{\lim}_{\Upsilon}[X//T]) \to T( \underleftarrow{\lim}_{ G_{P_j}}[X//T] \times_{\underleftarrow{\lim}_{\Upsilon'}[X//T]} \underleftarrow{\lim}_{L''}[X//T])$ every time we add a simple loop.  We can compute the contribution to the virtual class of $\underleftarrow{\lim}_{\Upsilon'}[X//T]$ of each of the components of $\underleftarrow{\lim}_{\Upsilon',k,O}[X//T]$ and so explicitly describe the kernel of $f$.

\end{remark}

I believe that the natural context in which to understand the virtual classes of the substrata in $\lim_ {G_{P_t}} [X//T]$ and the relations between them is an $n$-category naturally associated to the $T$--action on $X$. The objects in this category are cycles in the fixed loci of $T$. Let $c$ and $c'$ be two such cycles. A simple 1-morphism between them is given by the following data:

$\bullet $ An $(n-1)$-- dimensional  subtorus $T'$ of $T$.

$\bullet$ A fixed locus $X_{T'}$ for the $T'$--action on $X$, which contains both $c$ and $c'$.

$\bullet$ The union of all the orbits for the action of $T/T'$ on $X_{T'}$ that join $c$ and $c'$. We call this union the geometric presentation of $f$.

All 1-morphisms are compositions of simple 1-morphisms. The geometric presentation of such a 1-morphism is the union of the geometric presentations of the simple 1-morphisms involved.

Consider two 1-morphisms $f$ and $g$ from $c$ to $c'$. A simple 2-morphism from $f$ to $g$ is given by the following data:

$\bullet$ $T'' \hookrightarrow T$ subtorus of dimension $n-2$.

$\bullet$ A fixed locus $X_{T''}$ of $T''$ containing the geometric presentations of $f$ and $g$.

$\bullet$ All the orbits of the action of $T/T''$ on $X_{T''}$ whose boundaries are contained in the geometric presentations of $f$ and $g$. We call this union the geometric presentation of the 2-morphism.

All 2-morphisms are compositions of simple 2-morphisms. The geometric presentation of such a 2-morphism is the union of the geometric presentations of the simple 2-morphisms involved.

And similarly with 3-morphisms,..., $n$-morphisms. From this points of view, the moment map is a functor between $n$-categories, because a polytope divided in chambers leads also to a simple example of an $n$-category  by putting arrows on its faces corresponding to the flow of different $\CC^* \subset T$.

We would like to describe this $n$-category up to linear equivalence. The full description might get involved but here are some basic points:

$\bullet$ The objects of the $n$-category are replaced by the Chow ring $A(X_T)$ of all the fixed loci of the action of $T$ on $X$;

$\bullet$ Consider $T''\hookrightarrow T'\hookrightarrow T$ with $T''$ and $T'$ of dimensions $(k-1)$ and $k$, respectively. For any fixed locus $X_{T''}$ of $T''$ and for any fixed locus $Z$ of $T'$ in $X_{T''}$ we will construct two tautological rings $T(Z^+)$, and $T(Z^-)$, whose objects should be thought of as morphisms with target in $Z$, and, respectively, morphisms with source in $Z$. Let $BZ^+$ and $BZ^- \hookrightarrow X_{T''}$ be the unions of one dimensional orbits whose limit at $\infty$, respectively 0, is in $Z$. Let $I'$ be the set of indexes in $I$ corresponding to semistable conditions such that $U_i \cap BZ^+ \neq \emptyset $.

$ \underleftarrow{\lim}[BZ^+// (T/T')]$ is a weighted projective bundle over $\underleftarrow{\lim}[Z//T'']$. It could also be naturally be interpreted as defining a class in the ring $ T(\underleftarrow{\lim}_{I'} [X_{T''}//(T'/T'')]_i)). $
Let \bea  p^*_{I'}: T(\underleftarrow{\lim}_{I'} [X_{T''}//(T'/T'')]_i) \to T(\underleftarrow{\lim} [X_{T''}//(T'/T'')])\eea be the pull-back morphism described above. $p_{I'}^*([BZ^+// (T/T')])$ is independent of $i$ and it can be used to define \bea T(Z^+) = p_{I'}^*([B_Z^+// (T/T')])\cap  T(\underleftarrow{\lim} [X_{T'}//(T/T')]). \eea (Here $\cap$ denotes the cap product). Similarly we can define $T(Z^-)$.  One could proceed further to define tautological rings of morphisms with given source and given target.

\section{Functorial behavior of the virtual class}

We will denote by $A^{vir}(\underleftarrow{\lim}[X/T])$ the subgroup of the Chow group  $A(\underleftarrow{\lim}[X/T])$ obtained by intersection classes of  $A^*(\underleftarrow{\lim}[X/T])$ with $v(\underleftarrow{\lim}[X/T])$.

  We will start by proving that for any $T$- equivariant  morphism between two projective  smooth varieties $X$ and $Y$ with a torus action there is a morphism between the moduli spaces $\underleftarrow{\lim}[X/T]$ and $\underleftarrow{\lim}[Y/T]$.After this we will prove that there is a pull-back morphism between  $A^{vir}(\underleftarrow{\lim}[X/T])$ and $A^{vir}(\underleftarrow{\lim}[Y/T])$. Also we will check that there are natural push forward and pullback morphisms between the tautological rings:  $T(\underleftarrow{\lim}[X/T])$ and $T(\underleftarrow{\lim}[Y/T])$. This will be done by investigating the relation between the virtual classes of  $\underleftarrow{\lim}[X/T]$ and $\underleftarrow{\lim}[Y/T]$.

\begin{proposition}

Let $f:X \to Y$  be an equivariant morphism between $2$ smooth projective varieties. Let $L$ be a line bundle on $Y$ and $\Gamma(Y)$ a tree associated to the chamber decomposition of the image of the moment map given by $L$. There exists a line bundle $L'$ on $X$ and  a tree $\Gamma(X)$ such that there is a morphism between $\underleftarrow{\lim}_{\Gamma(X)}[X/T]$ and $\underleftarrow{\lim}_{\Gamma(Y)}[Y/T]$.

\end{proposition}

\begin{proof}

 For each point $y \in Y$ and $ x \in f^{-1}(y)$  the stabilizer group $Stab(y)$ is a subgroup of $Stab(x)$. In particular if $U$ is an open set in $Y$ such that for any point $y \in U$ has finite stabilizers and such that U has compact quotient $f^{-1}(U)$ will have the same properties. Let $L$ be an ample line bundle on $Y$ and let $K$ be a linearized line bundle on $X$ which is ample along the fibers of $f$. We can pick a very large $n$ such that $L'= L^n \otimes K$ is very ample.  Moreover for any linearization of $L$ such that the point zero in the image of the moment map is far enough from the walls if we denote by $U$ the open set of stable points associated to this linearization, the corresponding linearization for $L'$ will have as set of stable points $f^{-1}(U)$  as it can be checked by the numerical criteria.  However there are chambers in the image of the moment map for $L'$ which are not associated to chambers of $L$.  For any $y \in Y$ is $\CC^* $ the image of the $f^{-1}(y)$ under the moment map covers a union of chambers with the walls  between them corresponding to the chosen $\CC^*$. So for each wall $F_j \in J_Y$  in the image of the moment map of Y, there is a unique path $T_j$ in  $\Gamma(X)$ such that for any point $y$ which is mapped into $F_j$ under the moment map and for any $ x \in f^{-1}(y)$ the image of $x$ under the moment map will be in one of the chambers or walls associated to $T_j$. The components of the fixed locus of the $\CC^*$ in $X$ are inside the preimage of the fixed loci of $Y$. So all orbits in $X$ which are mapped by the moment map into the chambers associated to $T_j$ are in the preimage of the fixed loci in $Y$ which are mapped into $F_j$. In this way for any tree $\Gamma(Y)$ we have a canonical tree $\Gamma(X)$ associated to it.

To construct the morphism between  $\underleftarrow{\lim}_{\Gamma(X)}[X/T]$ and $\underleftarrow{\lim_{\Gamma(Y)}}[Y/T]$ we just need to check that there exists a natural morphism  $\underleftarrow{\lim}_{T_j}[X/T]$ to  $[Y/T]_j$ . The universal family $\cT$ over $\underleftarrow{\lim}_{T_j}[X/T]$  is a family of curves over $\cT/ \CC^*$ with a map into $Y$. After stabilizing  this family of curves we obtain a new family $\cT_s$ which corresponds to the data needed for a morphism into $[Y/T]_j$.

\end{proof}

\begin{theorem}

For any equivariant morphism $f : X  \to Y$ between two smooth projective varieties with a torus action $T$, the  morphism $\underleftarrow{\lim} f : \underleftarrow{\lim}[X/T] \to \underleftarrow{\lim}[Y/T]$ induces  morphisms $\underleftarrow{\lim} f_* : T( \underleftarrow{\lim}[X/T]) \to T(\underleftarrow{\lim}[Y/T])$ and  $\lim f^* : T( \underleftarrow{\lim}[X/T]) \to T(\underleftarrow{\lim}[Y/T])$ where the first is a morphism of groups and the last a morphism of rings. Also $\underleftarrow{\lim} f_* $ extends to a morphism defined on $ A^{vir}( \underleftarrow{\lim}[X/T])$. In particular if the morphism  $f : X  \to Y$ is birational   $\underleftarrow{\lim} f_*(v( \underleftarrow{\lim}[X/T]) = \underleftarrow{\lim}[Y/T]$.
\end{theorem}

\begin{proof}

We study first the case when $f$ is an embedding.  In this case the decomposition in chambers of the image of the moment map for $X$ is a subdivision of the decomposition for $Y$. We enlarge the category of quotients of $X$ by isomorphisms such that we obtain an isomorphism between the graph of chambers associated to $X$ and $Y$. Let $\Gamma$ be a maximal tree in associated to $Y$, and $ \Delta(Y) $ and $\Delta(X)$ be the diagonal of $\prod_{i  \in j \in \Gamma} [Y/T]_i$ respectively $\prod_{i  \in j \in \Gamma} [X/T]_i$ .
Let $C_Y$ be the cone of $\underleftarrow{\lim}_{\Gamma}[Y/T]$ in  $\prod_{ j \in \Gamma} [Y/T]_j$, embedded in $E(Y)$ the pullback of the normal bundle of $\Delta (Y)$ in  $\prod_{i  \in j \in \Gamma} [Y/T]_i$. The restriction of $C_Y$ to $ \underleftarrow{\lim}_{\Gamma}[X/T]$  contains the cone $ C(X)$ of  $\underleftarrow{\lim}_{\Gamma}[X/T]$ in  $\prod_{ j \in \Gamma} [X/T]_j$. Let $E(X)$ be the pullback of the normal bundle of $\Delta(X)$ in  $\prod_{i  \in j \in \Gamma} [X/T]_i$. Notice that the natural embeddings $C(X) \subset E(X)$,  $C(Y) \subset E(Y)$ and $E(X) \subset E(Y)$ and $C(X) \subset C(Y)$ form a commutative diagram. By pulling everything back to   $ \underleftarrow{\lim}[Y/T]$ we can define   $\underleftarrow{\lim} f^*$ and  $\underleftarrow{\lim} f_*$.  As pulling back and pushing forward classes between $ \underleftarrow{\lim}[Y/T]$ and  $ \underleftarrow{\lim}[X/T]$ can be done via pull back by pulling back and pushing forward classes between $\prod_{ j \in \Gamma} [Y/T]_j$ and $\prod_{ j \in \Gamma} [X/T]_j$, we conclude that $\underleftarrow{\lim} f^*$ and  $\underleftarrow{\lim} f_*$ can be restricted to tautological rings.

We consider now the case when the morphism $ f : X \to Y$  is smooth.  We will construct a  cone $C_Y$ inside a vector bundle $E$ on $\underleftarrow{\lim}[Y/T]$, of Chow class $v(\underleftarrow{\lim}[Y/T])$, and a cycle $C_X$ inside $ \underleftarrow{\lim} f^*(E)$ of class $v(\underleftarrow{\lim}[X/T])$ such that the natural morphism  $ f_E : \underleftarrow{\lim} f^*(E)  \to E$ maps $C_X$ into $C_Y$.This implies the existence of $\underleftarrow{\lim} f_*$.  We will check that there exists a class $\alpha \in A^*(\underleftarrow{\lim}[X/T]) $ such that $$f_{E*}(\alpha \cap [C_X])=[ C_Y].$$ By using this equality in conjunction with projection formula we can prove the existence of  $\underleftarrow{\lim} f^*$.

We will employ the notations from the previous proposition, with $\Gamma(Y)$ being a maximal tree and for each vertex $j$ of $\Gamma(Y)$, $T_j$ denotes the corresponding path in $\Gamma(X)$. We complete each of $\Gamma(X)$ to a maximal tree by completing each path $T_j$ to a bigger tree. We will continue to call them $\Gamma(X)$ and $T_j$ but they denote now a maximal tree, respectively a tree mapped to $j$. We will denote by $C(\underleftarrow{\lim_{T_j}}[X/T])$ a cycle in $\underleftarrow{\lim_{T_j}}[X/T]$ such that the class of $C(\underleftarrow{\lim_{T_j}}[X/T])$ is $v(\underleftarrow{\lim_{T_j}}[X/T])$. $N$ is the normal bundle of the diagonal $\Delta(Y)$ of $\prod_{i  \in j \in \Gamma(Y)} [Y/T]_i$, $E$ is the pullback of $N$ to $\underleftarrow{\lim}[Y/T]$. $C_Y$ is the restriction of the normal cone of $\underleftarrow{\lim_{\Gamma(Y)}}[Y/T] \to \prod_{j \in \Gamma} [Y/T]_j$ to $\underleftarrow{\lim}[Y/T]$. Let $C$ denote the normal cone of $\Delta(Y) \times_{\prod_{i  \in j \in \Gamma(Y)} [Y/T]_i}
\prod_{j \in \Gamma(Y)} C(\underleftarrow{\lim_{T_j}}[X/T])$ in $\prod_{j \in \Gamma(Y)} \underleftarrow{\lim_{T_j}}[X/T]$. There is a natural embedding of $C$ in the pull-back of $E$ which we will denote by $E'$. Let $E_{\underleftarrow{\lim_{\Gamma(X)}}[X/T]}$ and $C_{\underleftarrow{\lim_{\Gamma(X)}}[X/T]}$ denote the restriction of $E'$ and $C$ to $\underleftarrow{\lim_{\Gamma(X)}}[X/T]$. We further deform $E'$ to the normal cone of $E_{\underleftarrow{\lim_{\Gamma(X)}}[X/T]}$ in $E'$which induces a deformation of $C$ to $C'(X)$, the normal cone of $C_{\underleftarrow{\lim_{\Gamma(X)}}[X/T]}$. Note that the deformation of $E'$ splits now as a sum of two vector bundles: $E_{\underleftarrow{\lim_{\Gamma(X)}}[X/T]}$ and the pullback of the relative tangent bundle of the morphism $\Delta(X) \to \Delta(Y)$, which we will denote by $T_{\Delta(X)|\Delta(Y)}$  If $C(X)$ denotes the restriction of $C'(X)$ to $\underleftarrow{\lim}[X/T]$ by Lemma\ref{cone def} and the proof of Theorem 4.4, $[C(X)]$ as a class in the deformation of $E'$ is the same with $v(\underleftarrow{\lim}[Y/T])$. $C_X$ is a cycle in $C_{\underleftarrow{\lim_{\Gamma(X)}}[X/T]}$ restricted to $\underleftarrow{\lim}[Y/T]$ of class equal with $[C(X)]$. We should notice also that each time we discard some of the components of the normal cones by completing a local loop in $X$ an restricting to the inverse limit these components are contracted by the morphism into the corresponded space for $Y$.  In fact, instead of discarding the components of the normal cone, we can use Propositon 3.6 to contract them up to deformations.  First of all the deformation from Proposition 3.6 can be done in 2 steps: first it can be done for $Y$ and this induces a deformation for $X$, and then we can deform the deformation of $X$.   The contractions from Proposition 3.6 for $X$ either correspond to similar contractions for $Y$, in the case when closing a loop upstairs correspond to closing a loop downstairs, or this contractions identifies chains of orbits whose image is identical in $Y$ anyhow.

 We pick in each $A^*( \underleftarrow{\lim_{T_j}}[X/T])$ a cohomology class $\alpha_j$ such that $v(\underleftarrow{\lim_{T_j}}[X/T])\cap \alpha_j $ has as push-forward in $A_*([Y/T]_j)$ the fundamental class of $[Y/T]_j$.
 The pull back of $\prod_i \alpha_i \cup \prod_j \alpha_j$ to $C(X)$ will be denoted by $\alpha$ and has the property that $$f_{E*}(\alpha \cap [C_X])=[ C_Y].$$

As any morphism between smooth varieties can be factored in a closed embedding and a smooth morphism, we finished the
 proof.
\end{proof}

 We will continue to consider the smooth morphism between $X$ and $Y$ and restrict attention to an open sets  $U \subset \prod_{ j \in \Gamma} [X/T]_j$   such that the morphism between $\prod_{ j \in \Gamma} [X/T]_j$ and $\prod_{ j \in \Gamma} [Y/T]_j$ restricted to the $U_X$ and $U_Y$ is smooth.

 We consider $\cC_Y$ to denote the intrinsic normal cone of $\underleftarrow{\lim_{\Gamma(Y)}}[Y/T]$, (see \cite{behrend}) and the complex $E^{\bullet}= T \to N$ where $T$ is the pullback of the tangent bundle of $ \prod_{j \in \Gamma(Y)}$ to $\underleftarrow{\lim_{\Gamma(Y)}}[Y/T]$, and $ N$ is the pull-back of normal bundle of the diagonal $\Delta(Y)$. Following \cite{behrend} we define $h^1/h^0(E^{\bullet})$ as the stack quotient of $N$ by $T$. There is a natural embedding of $\cC_Y$ into $h^1/h^0(E^{\bullet})$.

From now on we will restrict all the stacks to $U$ whenever meaningful.
Preserving the notation from the proof of the previous theorem because of the way $U$ was chosen $ C = f_{\underleftarrow{\lim}}^*C_Y$ and its further deformation leads to $C'(X)= f_{\underleftarrow{\lim}}^*C_Y \times_ {\underleftarrow{\lim_{\Gamma(X)}}} C_2 $, where $C_2$ is the normal cone of $ \underleftarrow{\lim_{\Gamma(X)}}$ in $\underleftarrow{\lim_{\Gamma(Y)}}[Y/T]\times_{ \prod_{ j \in \Gamma} [Y/T]_j} \prod_{ j \in \Gamma} [X/T]_j$. We denote by $\cC_{\Gamma(X|Y)}$ the relative intrinsic normal cone of $ f_{\underleftarrow{\lim}}: \underleftarrow{\lim_{\Gamma(X)}}[X/T] \to \underleftarrow{\lim_{\Gamma(Y)}}[Y/T]$. We denote by $T^{\bullet}$ the complex given by the map from the relative tangent bundle  of $\prod_{ j \in \Gamma} [X/T]_j \to \prod_{ j \in \Gamma} [Y/T]_j$ restricted to $\underleftarrow{\lim_{\Gamma(X)}}[X/T]$ and the relative tangent bundle of  $\prod_{i \in j \in \Gamma} [X/T]_i \to \prod_{i \in j \in \Gamma} [Y/T]_i$.

The morphism between $X \to Y$ that we will consider will be of the form $Z \times \PP^N  \to \PP^N$, given by the projection on the second factor, and  the open set $U_Y$ corresponds to pairs of $T$ orbits which do not get contracted either when they are mapped in $Z$ or in $\PP^N$.

There exists an Artin stack  which we will denote $\underleftarrow{\lim_{G}}[BT] $, where $G$ is the graph associated to the chamber decomposition of the image of moment mapdefined on $\PP^N$.  $\underleftarrow{\lim_{G}}[BT] $ remembers the fibers of the universal family over   $\underleftarrow{\lim}[\PP^N/T] $ but forgets the morphism in $\PP^N$.  In other words there is smooth morphism  $\underleftarrow{\lim}[\PP^N/T] \to \underleftarrow{\lim_{G}}[BT] $. To check this we define  $\underleftarrow{\lim}[\PP^N/T]\times_{\underleftarrow{\lim_{G}}[BT] } \underleftarrow{\lim}[\PP^N/T] $ as the open set $U$ in  $\underleftarrow{\lim}[\PP^N \times \PP^N/T] $  such that for any $x \in U$ the universal family over $U$ is a union of $T$ orbits which are not contracted by projection to either $\PP^N$. The fact that $\underleftarrow{\lim_{G}}[BT] $ is an Artin stack is equivalent with the morphism $U \to  \underleftarrow{\lim}[\PP^N/T]$ being smooth. This is true as the the fiber of this morphism corresponds to a choice of $N+1$ equivariant sections in a line bundle. We should notice that for any partial graph $\Gamma$ corresponding to the inverse system associated to the action of $T$ on $\PP^N$ we could also create an Artin stack $\underleftarrow{\lim_{\Gamma}}[BT] $. We will give in fact the construction of this stack without to much reference to $\PP^n$. Let $ \mathcal{M}_{0,2,2}$ be the open substack of the Artin stack $\mathcal{M}_{0,2}$ of prestable genus zero curves with 2 marked points, consisting in those curves which have at most 2 components and have at least one marked point on each component. Notice that there is another open set in $\cM_{0,2}$ isomorphic to $B \CC^*= [point/\CC^*]$ corresponding to irreducible curves with 2 marked points. There is also a morphism from  $ \mathcal{M}_{0,2,2}$ to  $B \CC^*$ which forgets one of the components of any curve with 2 components. For each chamber $F_i$ in the image of the moment map for the $T$ action on $\PP^N$ we consider a copy of $BT_i=[point/T]$. For each wall $F_j$ let $\CC^*_j$ be the one parameter group associated to $F_j$ and $T_j=T/\CC^*_j$.  We denote by $[BT]_j:= [\cM_{0,2,2}/T_j]$. For $i \in j$ we have a morphism $[BT]_j \to [BT]_i$ induced by the morphism $ \mathcal{M}_{0,2,2} \to  B \CC^*$.  For each graph $\Gamma$ as above we can define $\underleftarrow{\lim_{\Gamma}}[BT] $ to be the inverse limit corresponding to the associated inverse system of maps between $BT_j$'s and $BT_i$'s.

\begin{theorem}

Let $Z$ be a smooth projective variety with an action of the torus $T$. Let $\Gamma$ be a maximal tree associated to the moment map of $Z$. Let $\cT$ be the universal family (with non-compact fiber)over $\underleftarrow{\lim_{\Gamma}}[Z/T]$. Let $e:\cT \to Z$ be the evaluation map, $\pi :\cT \to \underleftarrow{\lim_{\Gamma}}[Z/T]$ the projection map and $f: \underleftarrow{\lim_{\Gamma}}[Z/T] \to \underleftarrow{\lim_{\Gamma}}[BT]$ the forgetful map that forgets the morphism to $Z$. Let $\cC_Z$ denote the the relative intrinsic normal cone of $f: \underleftarrow{\lim_{\Gamma}}[Z/T] \to \underleftarrow{\lim_{\Gamma}}[BT]$. $R^{f}\pi_*(e^*T_Z)$ has cohomology only in degree $0$ and $1$ , where $T_Z$ is the tangent bundle of $Z$, $R^{f}\pi_*$ is the fixed part of the direct image.  There is an embedding of $\cC_Z $ into $h^1/h^0(R^{f}\pi_*(e^*T_Z))$. The pull back of $vir(\underleftarrow{\lim_{\Gamma}}[BT])$ via the class of $\cC_Z$ in  $h^1/h^0(R^{f}\pi_*(e^*T_Z))$ is $vir( \underleftarrow{\lim_{\Gamma}}[Z/T] )$.
\end{theorem}

\begin{proof}
Let $U$ denote as before the open set in $ \underleftarrow{\lim_{\Gamma}}[\PP^N \times Z/T]$. Let $u: U \to \underleftarrow{\lim_{\Gamma}}[Z/T]$ denote the natural morphism that forgets the morphism into $\PP^N$, and $\bar{u}$ the lift of $u$ to the universal families.
Because of the discussion above the only thing left to check is that $T^{\bullet}=u^*(h^1/h^0(R^{f}\pi_*(e^*T_Z)))=h^1/h^0(R^{f}\pi_{U*}(\bar{u}^*(e^*T_Z))$ in the derived category, where $\pi_U$ denotes the projection of the universal family over $U$ to $U$.  From its construction $\cT$ the universal family over $U$ is naturally covered with the open sets $\cT_j$ which are the the pull-back of the universal families over $[Z \times \PP^N/T]_j$.   Notice that for each $i \in j $ there is an embedding $\cT_i \hookrightarrow \cT_j$. For any sheaf $E$ on $\cT$ there is an exact sequence: $$0 \rightarrow E \rightarrow \oplus_{j \in \Gamma}E_{|\cT_j} \rightarrow \frac{\oplus_{i \in j \in \Gamma}E_{|\cT_i}}{\oplus_{i \in \Gamma}E_{|\cT_i}} \rightarrow 0,$$ where the morphisms are given by restriction map. The morphism $\pi_U$ restricted to $\cT_j$ or $\cT_i$ is affine so $$R^{f}\pi_{U*}(E) =\pi_{|U*}^f(  \oplus_{j \in \Gamma}E_{|\cT_j}) \rightarrow \pi_{|U*}^f(\frac{\oplus_{i \in j \in \Gamma}E_{|\cT_i}}{\oplus_{i \in \Gamma}E_{|\cT_i}}),$$ in the derived category. In the case $E=\bar{u}^*(ev^*(T_Z))$ this complex is exactly $T^{\bullet}$.

\end{proof}

\section{Further developments}

 Could the constructions in this article be extend to the case when we replace the action of a torus  by the action of a reductive group ?

There are some remarkable differences between acting with a torus and acting with a reductive group. In the last case we could not expect our moduli space to be inverse limit of $GIT$ quotients. We will give an explicit example of this situation. Let $\PP^n_d$ denote the projective space of dimension $(n+1)(d+1)-1$ thought as $(n+1)$-tuples of homogeneous polynomials in two variables of  degree $d$
 up to multiplication by scalers. This space is a compactification of the space of degree $d$ maps from $\PP^1$ to $\PP^n$ and as such it has a natural $Sl_2$ action. If we are interested in the degeneration of all  $SL_2$ orbits inside  $\PP^n_d$ the natural moduli space we construct is the moduli space of stable maps $\ol{M}_{0, 0}(\PP^n, d)$. From this point of view the universal family over  $\ol{M}_{0, 0}(\PP^n, d)$ is the graph space  $\ol{M}_{0, 0}(\PP^n \times \PP^1, (d,1))$.  $\ol{M}_{0, 0}(\PP^n, d)$ is different from the $GIT$ quotient of $\PP^n_d$ by $Sl_2$. Also the $GIT$ quotient is unique so there is no variation of $GIT$ that we can speak about. The birational transformations that relate  $\ol{M}_{0, 0}(\PP^n \times \PP^1, (d,1))$ and $\PP^n_d$ are described in \cite{noi1} and can be restated completely in terms of the action of $ SL_2$. These birational transformations are very much related with the constructions in this paper as it can be seen from Poposition 2.3. My hope is that this type of constructions can be extended to a large class of varieties with reductive group action.

 For example, similar to $ \PP^n_d$ one can consider the space $\PP^n_{d,k}$ where we replace polynomials in $2$ variables by polynomials in $k+1$ variables. Success in constructing such a moduli space for the action of $SL(k+1)$ on  $\PP^n_{d,k}$, could lead to a definition of a moduli space of stable maps between projective spaces together with an intersection theory on these moduli spaces.


\providecommand{\bysame}{\leavevmode\hbox to3em{\hrulefill}\thinspace}

\end{document}